\documentclass{article}
\usepackage{color}
\usepackage{amsfonts}
\usepackage{tikz}
\usepackage{tikz-feynman}
\usepackage{amssymb}
\usepackage{amsmath}
\usepackage{circuitikz}
\usepackage{amsthm}
\usepackage[mathscr]{euscript}
\usepackage{hyperref}
\usepackage{geometry}
\usepackage{comment}
\usepackage{float}  % 需要先引入这个包
\usetikzlibrary{matrix}%%使用矩阵

\usetikzlibrary{er,positioning}
\newtheorem{theorem}{Theorem}[section]
\newtheorem{corollary}{Corollary}[section]
\newtheorem{definition}{Definition}[section]
\newtheorem{lemma}{Lemma}[section]
\newtheorem{proposition}{Proposition}[section]
\newtheorem{remark}{Remark}[section]

\newtheorem{assump}{Assumption}[section]
\numberwithin{equation}{section}

\newcommand{\m}{\mathbb}
\newcommand{\ml}{\mathcal}
\newcommand{\mk}{\mathfrak}
\newcommand{\p}{\partial}

\makeatletter
\newcommand\RSloop{\@ifnextchar\bgroup\RSloopa\RSloopb}
\makeatother
\newcommand\RSloopa[1]{\bgroup\RSloop#1\relax\egroup\RSloop}
\newcommand\RSloopb[1]%
{\ifx\relax#1%
	\else
	\ifcsname RS:#1\endcsname
	\csname RS:#1\endcsname
	\else
	\GenericError{(RS)}{RS Error: operator #1 undefined}{}{}%
	\fi
	\expandafter\RSloop
	\fi
}
\newcommand\X{0}
\newcommand\RS[1]%
{\begin{tikzpicture}
		[every node/.style=
		{circle,draw,fill,minimum size=1.5pt,inner sep=0pt,outer sep=0pt},
		line cap=round
		]
		\coordinate(\X) at (0,0);
		\RSloop{#1}\relax

	\end{tikzpicture}
}
\makeatother
\newcommand\RSdef[1]{\expandafter\def\csname RS:#1\endcsname}
\newlength\RSu
\RSdef{i}{\draw (\X) -- +(90:\RSu) node{};}
\RSdef{l}{\draw (\X) -- +(135:\RSu) node{};}
\RSdef{r}{\draw (\X) -- +(45:\RSu) node{};}
\RSdef{I}{\draw (\X) -- +(90:\RSu) coordinate(\X I);\edef\X{\X I}}
\RSdef{L}{\draw (\X) -- +(135:\RSu) coordinate(\X L);\edef\X{\X L}}
\RSdef{R}{\draw (\X) -- +(45:\RSu) coordinate(\X R);\edef\X{\X R}}
\RSdef{n}{\draw node{};}
\RSdef{c}{\draw node[fill=none]{};}
\begin{document}
	\author{ShuoLin Zhang,\quad Zhaonan Luo,\quad Zhaoyang Yin}
	\title{A Generalized Framework for Singular Fractional Burgers Equations with Stochastic Forcing }
	\date{}
	\maketitle
	\tableofcontents
	\begin{abstract}
		In this paper, we investigate the regularization effect of fractional stochastic forcing on Burgers-type equations with fractional dissipation, with an application to the Degasperis--Procesi (DP) equation. In particular, we consider the perturbation induced by the singular noise $|D|^{1/2}\xi$ and establish local well-posedness in the negative Sobolev space $H^{-1/4+\delta}$ for some small $\delta>0$. Due to the singular nature of the nonlinear interactions, classical solution theories cannot be directly applied.
\par Inspired by the framework developed in \cite{hairer2013solving,gubinelli2017kpz}, we introduce a generalized solution theory based on an enhanced structure and derive the effective equations satisfied by these generalized solutions. Our main contribution is the establishment of a general framework for describing singular PDEs driven by rough data. Moreover, we prove the convergence of the associated non-Gaussian rough structures in the fractional dissipation setting. As an application, we apply this framework to the stochastic Degasperis--Procesi equation and obtain its local well-posedness in the low-regularity regime.
	\end{abstract}
\section{Introduction}
The Degasperis-Procesi (DP) equation 
\begin{equation}\label{eq;DP equation;no dissipative}
		\p_tm+u\p_xm+3\p_xum=0,\quad m(t,x)=(1-\partial_x^2)u(t,x),\quad u(0,x)=u_0(x).
	\end{equation}
or equivalently, in terms of the the fluid velocity $u(t,x):=g\ast m:=(1-\p_x^2)^{-1}m$ and rewrite the DP equation as
	\begin{equation}\label{eq;DP equation;general}
		\p_tu+\frac{1}{2}\p_x(u^2)+\frac{3}{2}(1-\partial_x^2)^{-1}\partial_x(u^2)=0,\quad u(0,x)=u_0(x).
	\end{equation}
   The Degasperis--Procesi (DP) equation was introduced by Degasperis and Procesi 
\cite{degasperis1999asymptotic,degasperis2002new}  through the method of asymptotic integrability. It serves as a model for nonlinear shallow water wave propagation and 
shares the same asymptotic accuracy as the Camassa--Holm equation \cite{CAMASSA19941}. 
 	 \par The following citations \cite{degasperis1999asymptotic,degasperis2002new} provide physical background for this class of equations; a more systematic discussion can be found in \cite{ai2010global}. From the mathematical perspective,  local well-posedness of \eqref{eq;DP equation;no dissipative} for initial data in $H^s$ with $s>\frac{3}{2}$ was established in \cite{liu2006global}, together with sharp blow-up criteria and finite-time singularity results in \cite{liu2007blow}.
\par In \cite{ai2010global}, Ai and Gui established the local well-posedness of equation \eqref{eq;DP equation;general} in the case $\gamma=2$, with solutions in $H^s$ for $s>0$. However, the extension to the negative regularity regime remains largely open, because the nonlinear transport term is no longer well-defined in the classical distributional framework.
The singular structure of the nonlinear transport term in the Degasperis--Procesi (DP) equation is essentially the same as that of the Burgers nonlinearity. This observation motivates us to develop an abstract framework for a broader class of fractional Burgers-type equations, within which the DP equation can be treated as a particular example.
\par One approach to studying solutions with negative regularity is provided by the theory of random initial data, which has received considerable attention in recent years; see, for instance, \cite{burqRandomDataCauchy2008a,burqRandomDataCauchy2008c,deng2022random,deng2024probabilistic,deng2024invariant}. More precisely, in \cite{chen2024probabilistic}, the first author, together with Chen and Gao, studied the DP equation and the fractional Burgers equation with initial data in \(H^{-\alpha}\), for \(\alpha>0\) and \(\gamma\in(5/4,3/2)\). In this paper, we present a different approach: the presence of suitable stochastic forcing enables us to establish well-posedness for the equation in distribution spaces of negative regularity. The key idea is to introduce an enhanced structure that captures the singular interactions generated by the noise and provides a meaningful interpretation of the nonlinear terms.
\par Before presenting our main results, we briefly review the role of noise in regularizing singular PDEs and the framework of SPDEs for handling ill-defined nonlinear interactions.
\subsection{Noise-induced regularization}
\par Deterministic partial differential equations are often derived under idealized assumptions that neglect unresolved microscopic fluctuations and external random influences. In many applications, such effects are neither negligible nor amenable to explicit modeling at the macroscopic level. A standard approach to incorporate such uncertainty is to introduce stochastic forcing terms, leading to stochastic partial differential equations (SPDEs) of the form
\begin{equation}
    \p_t u=Lu+\ml N(u)+\xi
\end{equation}
typically $\xi$ represents a space-time noise. From a mathematical standpoint, SPDEs extend the classical deterministic framework by allowing the driving terms to be generalized stochastic processes, thereby providing a rigorous setting for studying dynamics under random perturbations; see \cite{walsh2006introduction} and \cite{da2014stochastic}. One primary motivation for such models is that randomness may be interpreted as an effective description of unresolved degrees of freedom, consistent with the fluctuation–dissipation principle in statistical physics \cite{kubo1966fluctuation}.
\par In addition to their modeling relevance, stochastic perturbations can improve the analytical properties of the underlying equations. In various settings, noise has been shown to prevent finite-time blow-up, enhance well-posedness, or induce regularizing effects on solutions in ODEs and PDEs (see, e.g., \cite{catellier2016rough,catellier2016averaging,athreya2024well,flandoli2010well}).  Moreover, SPDEs often admit invariant measures and exhibit ergodic behavior, which allows for a probabilistic characterization of long-time dynamics beyond pointwise trajectories.
\par In the present paper,  the regularization effect refers not to a direct improvement of the regularity of solutions, but to the fact that stochastic forcing provides an enhanced structure allowing singular nonlinearities to be rigorously interpreted.  Consider the DP equation with fractional disspation driven by $|D|^{1/2}\xi$ on torus $\mathbb T=\mathbb R/2\pi\mathbb Z$ 
\begin{equation}\label{eq;DP;Burgers-like}
    \p_tu-\Lambda^\gamma u+\frac12\partial_x(u^2)
    +\frac32(1-\partial_x^2)^{-1}\partial_x(u^2)
    =|D|^{1/2}\xi,
    \qquad u(0,x)=u_0(x),
\end{equation}
for $\gamma>3/2$, where $\xi$ is a space-time white noise satisfying
\begin{equation}\label{eq;white noise}
\mathbb E[\xi(\omega,t,x)\xi(\omega,s,y)]
=\delta(t-s)\delta(x-y),
\end{equation}
for $t,s\geq0$ and $x,y\in\mathbb T$.
\par Remarkably, we show that the presence of space-time white noise provides a rigorous interpretation of the equation for a non-trivial class of initial data in negative Sobolev spaces $H^s$ beyond the deterministic well-posedness regime. The noise does not simply smooth the solution; rather, it generates an enhanced algebraic structure that allows us to interpret the singular nonlinear interactions and establish local well-posedness in a regime where the deterministic equation is not classically defined.
\par  A key feature of the DP equation is the presence of a nonlinear transport term, which shares the same singular structure as the Burgers nonlinearity. Therefore, instead of directly studying the DP equation, we first investigate  fractional singular Burgers equation
\begin{equation}
    \p_t u-\Lambda^\gamma u=\p_x(u)^2+|D|^{1/2}\xi,
\end{equation}
which captures the essential singular nonlinear mechanism of the DP equation. The well-posedness theory developed for the generalized Burgers framework can then be applied to the DP equation. 
\par Before presenting our main results, we first clarify the interpretation of standard deterministic PDEs driven by noise in both the physical and mathematical sense, and briefly recall how singularities are handled in the theory of SPDEs.
\subsection{Singularity in SPDEs}
\par The singularity of nonlinear terms in stochastic PDEs driven by rough noise has become a central difficulty in the theory of singular SPDEs. We begin with one of the most representative examples, the Kardar--Parisi--Zhang (KPZ) equation introduced in \cite{kardar1986dynamic},
\begin{equation}
    \partial_t h-\Delta h=(\partial_x h)^2+\xi ,
\end{equation}
which describes the evolution of randomly growing interfaces. Due to the irregularity of the solution, the nonlinear term $(\partial_x h)^2$ is not well-defined in the classical sense, making the equation ill-posed without further interpretation.
\par Over the past decades, various approaches have been developed to construct and understand solutions of the KPZ equation, including the Cole--Hopf transformation, energy solutions, and probabilistic approaches related to universality; see, for example, \cite{bertini1997stochastic,gonccalves2014nonlinear,alberts2014intermediate}. A decisive breakthrough was achieved by Hairer \cite{hairer2013solving}, who introduced the theory of regularity structures and provided a robust framework for singular stochastic PDEs through renormalization and fixed-point arguments. Related ideas were further developed through paracontrolled calculus and have been successfully applied to other singular equations, including Rough Burgers equations; see \cite{gubinelli2015paracontrolled,gubinelli2017kpz,catellier2018paracontrolled,gubinelli2023paracontrolled,hairerTheoryRegularityStructures2014}.
\par To briefly recall the development of singular KPZ-type equations, let 
$\xi_\epsilon$ be a smooth approximation of the space-time white noise $\xi$ obtained by convolution with a smooth mollifier,
\[
\xi_\epsilon=\rho_\epsilon *\xi,
\qquad 
\rho_\epsilon=\epsilon^{-1}\rho(\epsilon^{-1}\cdot).
\]
The corresponding renormalized KPZ equation is given by
\begin{equation}
    \partial_t h_\epsilon-\Delta h_\epsilon
    =(\partial_x h_\epsilon)^2-C_\epsilon+\xi_\epsilon ,
\end{equation}
where $C_\epsilon$ denotes the renormalization constant.
\par A fundamental contribution of Hairer was to provide a rigorous interpretation of the singular nonlinear term $(\partial_x h)^2$ by introducing an enhanced structure for the irregular component of the solution. More precisely, the solution is decomposed into a rough part and a regular remainder,
\[
h=\sum_{\tau\in\ml T'}c(\tau)X^\tau+v,
\]
where the collection $\{X^\tau\}$ encodes the singular interactions generated by the noise and $v$ represents the regular component. Under this framework, the ill-defined nonlinear equation can be transformed into a well-defined equation for the remainder term, which allows one to establish the convergence of the approximating solutions and the well-posedness of the limiting equation.
\par Following this idea, similar approaches have been developed for Rough Burgers equations. In particular, Gubinelli and Pekowoski \cite{gubinelli2017kpz} studied the Rough Burgers equation
\begin{equation}
    \partial_t u-\Delta u=\partial_x(u^2)+D\xi,
\end{equation}
as well as the KPZ equation on the torus $\mathbb T$. Moreover, the energy solution framework on $\mathbb R$ was further developed in \cite{gubinelli2018energy}.
\par Inspired by the theory developed for singular stochastic PDEs such as the KPZ equation, we aim to construct a general framework for nonlinear equations with rough stochastic forcing. Unlike regularity structures, our approach does not aim to construct a universal model space. Instead, we formulate a generalized equation directly driven by enhanced stochastic objects. In order to unify the treatment of the fractional Burgers equation and the DP equation, we introduce the Singular Fractional Burgers-type(SFB) equation
\begin{equation}\label{eq;generalized Burgers-type}
    \partial_t u-\Lambda^\gamma u=f(D)(u^2)+|D|^{1/2}\xi .
\end{equation}
The operator $f(D)$ represents a first-order Fourier multiplier and covers both the local transport operator in Burgers equations and the nonlocal transport operator appearing in the DP equation. The main result is stated as follows.
\subsection{Main results}
\begin{theorem}\label{th;main result}
Let $\gamma\in(\frac32,2]$ and $f(D)=g(D)D$, where $g$ satisfies Assumption \ref{ass:g}. Let $Y$ be the stationary process obtained as the limit of $Y_\epsilon$ defined in \eqref{eq;Y;app}. For sufficiently small $\delta>0$, let $u_0\in W^s$ with $\frac14-\delta<s<\frac14+\delta$, and set $u_{\mathrm{in}}=u_0+Y(0)$. Consider the approximating equation
\begin{equation}\label{eq;Burgers;app}
    \p_t u^\epsilon-\Lambda^\gamma u^\epsilon=f(D)((u^\epsilon)^2)+|D|^{\frac12}\xi_\epsilon,\quad u^\epsilon(0)=\phi_\epsilon\ast u_{\mathrm{in}}.
\end{equation}
Then there exist almost surely positive random times $T^*$ and $T_\epsilon$, with $T_\epsilon\to T^*$ in probability, and a process $u\in C_{T^*}W^{-\frac14+\delta}$ almost surely, with $u(0)=u_{\mathrm{in}}$, such that
\[
\Vert u^\epsilon-u\Vert_{C_{T_\epsilon\wedge T^*}W^{-\frac14+\delta}}\to0
\]
in probability as $\epsilon\to0$.
\end{theorem}
\begin{remark}
The choice $g(D)=1$ gives the fractional singular Burgers equation. The choice $g(D)=-\frac12-\frac32(1-\partial_x^2)^{-1}$ recovers the stochastic Degasperis--Procesi equation. Therefore, the above theorem provides a unified framework for both equations.
\end{remark}
\par We point out that Jara and Gubinelli also consider the fractional Rough Burgers equation
\begin{equation}
    \p_t u-\Lambda^\gamma u=\p_x(u^2)+|D|^{\gamma/2}\xi.
\end{equation}
They prove the existence for $\gamma>1$ by the regularization effect of noise and the uniqueness for $\gamma>\frac{5}{2}$. In our setting, the noise is not such low regularity but the solution has uniqueness and continuously dependence. In fact, we restrict our attention to $\gamma>\frac
32$ in order to focus on the main ideas; the more singular case $\gamma\leq\frac{3}{2}$ requires additional tech such as the rough path theory or paracontrolled method and is not treated here. 
\par It should be noted that the solution $u$ constructed in Theorem \ref{th;main result} is not a classical solution to in the usual sense. Due to the singular nature of the nonlinear interactions, the generalized solution cannot be interpreted as satisfying the original equation directly. This naturally leads to the question of identifying the effective equation satisfied by $u$.  It can be seen as the Generalized of original equation, hence we called Generalized Singular fractional Burgers-type equation 
We introduce a method based on a fixed‑point argument for a generalized equation driven by enhanced data. The proof of Theorem \ref{th;main result} can be divided into two steps. In the first step, we construct a generalized rough Burgers equation and prove its well‑posedness. In the second step, we establish the convergence of the non‑Gaussian rough path to the enhanced data; the convergence in Theorem \ref{th;main result} then follows from a continuous dependence argument. Compared with a direct approach in \cite{gubinelli2017kpz}, our method offers two advantages: (1) it provides an equation satisfied by the generalized solution, and (2) it allows us to study the generalized solution directly using PDE theory, thereby reducing the complexity inherent in asymptotic analysis. For example, we prove an energy estimate, though at the cost of a high regularity requirement on $\gamma$.

\subsection{Organized of the paper}
\par This paper is organized as follows. In Section \ref{sec;basic tools}, we collects preliminary tools. In Section \ref{sec;GSFB} introduces the generalized framework. In Section \ref{sec;well-posedness} proves the well-posedness result. In Section \ref{sec;high order} establishes the convergence of stochastic enhanced data.
\section{Basic tools}\label{sec;basic tools}
We first recall some basic facts concerning Littlewood-Paley theory and Besov spaces. We introduce the nonhomogeneous Littlewood-Paley decomposition $\{\Delta_j\}_{j\geq -1}$ such that for any$u\in\mathcal S'(\mathbb T^d)$ (the same construction is valid on
$\mathbb R^d$), we have
\[
u=\sum_{j\geq -1}\Delta_j u .
\]

More precisely, the dyadic blocks $\{\Delta_j\}_{j\geq -1}$ are defined through the Fourier transform.
Let $\chi$ and $\rho$ be smooth radial functions taking values in $[0,1]$
such that
\[
\chi(k)+\sum_{j\geq0}\rho(2^{-j}k)=1,
\qquad k\in\mathbb Z^d,
\]
where $\chi$ is supported in a neighborhood of the origin and $\rho$ is supported in an annulus. Then we define
\[
\mathcal F(\Delta_{-1}u)(k)=\chi(k)\hat u(k),
\]
and
\[
\mathcal F(\Delta_j u)(k)=\rho(2^{-j}k)\hat u(k),
\qquad j\geq0 .
\]

For $s\in\mathbb R$ and $1\leq p,r\leq\infty$, the Besov space $B^s_{p,r}$ is defined by the norm
\[
\|u\|_{B^s_{p,r}}
=
\left(
\sum_{j\geq-1}
(2^{js}\|\Delta_j u\|_{L^p})^r
\right)^{1/r},
\]
with the usual modification when $r=\infty$. In particular, we define $\ml C^s=B_{\infty,\infty}^s$ for $s\in\m R$, then for $s\in(0,1)$, this space coincides with the classical Hölder space:
\[
\mathcal C^s
=
\left\{
f\in L^\infty:
\sup_{x\neq y}
\frac{|f(x)-f(y)|}{|x-y|^s}<\infty
\right\}.
\]
Throughout the paper, we use the notation
\[
W^s=\mathcal C^s\cap H^s,
\qquad s\in\mathbb R .
\]
For a Banach space $X$, we denote
\[
C_TX=C([0,T];X)
\]
with the norm
\[
\|f\|_{C_TX}
=
\sup_{t\in[0,T]}\|f(t)\|_X .
\]

For $s>0$, both $W^s$ and $C_TW^s$ are Banach algebras, namely,
\[
\|fg\|_{W^s}\lesssim \|f\|_{W^s}\|g\|_{W^s},\quad \Vert fg\Vert_{C_TW^s}\lesssim \Vert f\Vert_{C_TW^s}\Vert g\Vert_{C_TW^s}.
\]

We will repeatedly use the following estimate for heat flow.

\begin{lemma}[\cite{zhang2025resultsfractionalroughburgers}]
\label{le;est;heat flow}
Let $\gamma\geq1$, $\delta\in(0,\gamma)$ $\alpha\in\mathbb R$, and define $\mathcal L=\partial_t-\Lambda^\gamma$. Assume that $f(0)\in W^{\alpha+\delta} $, $\mathcal Lf\in C_T W^\alpha$ for $T>0$ then for $T\leq 1$, we have
\[
\|f(t)\|_{C_TW^{\alpha+\delta}}
\leq
C\|f(0)\|_{W^{\alpha+\delta}}
+
CT^{1-\frac{\delta}{\gamma}}
\|\mathcal Lf\|_{C_TW^\alpha}.
\]
\end{lemma}
Let $g(D)$ be a Fourier multiplier operator defined by
\[
\widehat{g(D)u}(k)=g(k)\hat u(k),
\qquad k\in\mathbb Z.
\]
We assume that $g(D)$ is a zero-order Fourier multiplier satisfying the following properties
\begin{assump}\label{ass:g}
The Fourier symbol $g(x)\in C^1$ satisfies following properties
\begin{itemize}
    \item { $g(x)$ is even and
    \[
|g(k)|\lesssim 1,\qquad x\in\mathbb R.
\]}
\item{
\begin{align*}
    |g'(x)|\lesssim 1,\quad|xg'(x)|\lesssim1.
\end{align*}
}
\item {$g(D)$ is bounded on $W^s$.}
\end{itemize}
\end{assump}
We now verify that the operators arising from the Burgers equation and the Degasperis--Procesi equation satisfy Assumption \ref{ass:g}. For the Burgers equation, we have \(g(D)=1\), so the assumption is trivially satisfied. For the DP equation, the corresponding operator is $g(D)=-\frac{1}{2}-\frac{3}{2}(1-D^2)^{-1}.$ Its Fourier symbol is
\[
g(x)=-\frac12-\frac{3}{2(1+x^2)},
\]
which is smooth and even. Moreover,
\[
|g(x)|\leq2,\qquad g'(x)=\frac{3x}{(1+x^2)^2}.
\]
Consequently,
\[
|g'(x)|=\frac{3|x|}{(1+x^2)^2}\leq\frac32,
\qquad
|xg'(x)|=\frac{3x^2}{(1+x^2)^2}\leq\frac34,
\qquad x\in\mathbb R.
\]
Thus, the first two conditions in Assumption \ref{ass:g} are satisfied.

\par The boundedness of \(g(D)\) on Sobolev spaces \(H^s\) is standard. To prove its boundedness on \(\ml C^s\), notice that

$$
(1-D^2)^{-1}u=K*u,
$$

where

$$
K(x)=\sum_{k\in\mathbb Z}\frac{e^{ikx}}{1+k^2}\in L^1(\mathbb T).
$$

Therefore, for every dyadic block \(\Delta_j u\), Young's convolution inequality yields

$$
\begin{aligned}
\|(1-D^2)^{-1}\Delta_j u\|_{L^\infty}
&=\|K*\Delta_j u\|_{L^\infty}\\
&\lesssim \|K\|_{L^1}\|\Delta_j u\|_{L^\infty}
\lesssim \|\Delta_j u\|_{L^\infty}.
\end{aligned}
$$

Multiplying by \(2^{js}\) and taking the supremum over \(j\), we obtain

$$
\|(1-D^2)^{-1}u\|_{\ml C^s}
\lesssim
\|u\|_{\ml C^s}.
$$

Thus, \(g(D)\) is bounded on \(\ml C^s\). Combining this with its boundedness on \(H^s\) and the relation$W^s=\ml C^s\cap H^s,$ we conclude that \(g(D)\) is bounded on \(W^s\). Hence, the operator associated with the DP equation satisfies Assumption \ref{ass:g}.
\subsection{Bony's para-product}\label{subsec;Littlewood-Paley}
Let $f$ and $g$ belongs to $\ml S'$, by Littlewood-Paley decomposition, it can be present as 
$$f=\sum_{j\geq -1} \Delta_j f\quad and \quad g=\sum_{j\geq -1} \Delta_j g$$
where $\Delta_j$ is a Fourier operator which is support on $2^j\ml C$ for some fixed annual $\ml C$ and any $j\geq 0$ , and support on a Ball for $j=-1$, see more details in \cite{bahouri2011fourier,zhang2025resultsfractionalroughburgers}. Then we define the \textbf{nonhomogeneous paraproduct} between $f$ and $g$ as 
\begin{equation}
	f\prec g=\sum_{j}S_{j-1} f\Delta_jg,
\end{equation}
and the nonhomogeneous remainder of $f$ and $g$ as 
\begin{equation}
	f\circ g=\sum_{\vert i-j\vert\leq 1}\Delta_if \Delta_jg.
\end{equation}
Let $f\succ g=g\prec f$, then its easy to check that $fg$ has following decomposition
\begin{equation}
	fg= f\prec g+f\succ g+f\circ g.
\end{equation}

\begin{lemma}[\cite{bahouri2011fourier}]\label{le;est;paraproduct}\label{le;est;paraproduct;modify}
	Define $f\prec g:=\sum_{j}S_{j-1}f\Delta_jg$, $f\succ g=g\prec f$, $f\circ g:=\sum_{\vert k-j\vert\leq 1}\Delta_k f\Delta_j g$, for all $k\in \mathbb{N}$, we have following estimate
	\begin{itemize}
		\item{1.  $$\Vert f\prec g\Vert_{W^s}\leq C\Vert f\Vert_{L^\infty}\Vert  g\Vert_{W^{s}}.$$}
		\item{2. If $t<0$ and $\frac{1}{r}=\min(1,\frac{1}{r_1}+\frac{1}{r_2})$, then $$\Vert f\prec g\Vert_{W^{s+t}}\leq C\Vert f\Vert_{B_{\infty,r_1}^t}\Vert  g\Vert_{W^{s}}.$$}
		\item{3. If $\frac{1}{p}=\frac{1}{p_1}+\frac{1}{p_2}\leq1$, $\frac{1}{r}=\frac{1}{r_1}+\frac{1}{r_2}\leq1$ and $s=s_1+s_2>0$, then
			$$\Vert f\circ g\Vert_{W^{s}}\leq C\Vert f\Vert_{W^{s_1}}\Vert g\Vert_{W^{s_2}}.$$}
	\end{itemize}
\end{lemma}
\subsection{Contraction mapping theorem}
Next we introduce a generalized contraction mapping theorem which compared with Lemma 5.5($ L=0$) in \cite{bahouri2011fourier}.
\begin{lemma}\label{le;Picard th}\label{le;continuous dependence}
	Let $E$ be a Banach space, and $\ml B$ a continuous bilinear map from $E\times E$ to $E$, and $r$ a positive real number such that
	$$r<\frac{1}{8\Vert \ml B\Vert}\quad with \quad\Vert \ml B\Vert=\sup_{\Vert u\Vert_E,\Vert v\Vert_E\leq1}\Vert \ml B(u,v)\Vert_E.$$ 
	Let $ L (x)$ is a linear map $E$ to $E$ and satisfied 
    \begin{equation}\label{eq;condition L}
     \Vert L\Vert_{E\to E}=\sup_{\Vert v\Vert_E\leq 1}\Vert L(v)\Vert_E<\frac{1}{4}.
    \end{equation}
	Then for any $\Vert a\Vert_E<r$, there exists a unique $\tilde x$ in $\ml B(0,2r)$ such that 
	$$\tilde x=a+ L(\tilde x)+\ml B(\tilde x,\tilde x).$$
    Moreover, if there exists $a_n$ and $L_n$ satisfying
       \begin{equation}\label{eq;an,Ln}
        \Vert a_n-a\Vert_E\to 0,\quad \Vert L_n-L\Vert_{E\to E}\to0,\quad as\: n\to \infty.
    \end{equation}
    then for $x_n=a_n+L_n(x)+\ml B(x_n,x_n)$, there exists $N\in \m N$ such that for any $n>N$, we have
    \begin{equation}
        \Vert x_n-x\Vert_E\leq 4\Vert a_n-a\Vert_E+8r\Vert L_n-L\Vert_{E\to E}.
    \end{equation}
    \begin{proof}
        The first result is from \cite{zhang2025resultsfractionalroughburgers}, for the second result, by \eqref{eq;an,Ln}, there exists $N\in\m N$, such that for $n>N$, $\Vert L_n\Vert_{E\to E}<1/4$, then by the first results, there exists a unique $x_n$ in $B(0,2r)$ satisfying $x_n=a_n+L_{n}(x_n)+\ml B(x_n,x_n)$. Moreover, we have
       \begin{align*}
            \Vert x_n-x\Vert_E< \Vert a_n-a\Vert_E+\frac{3}{4}\Vert x_n-x\Vert_E+2r\Vert L_n-L\Vert_{E\to E}.
        \end{align*}
        then we have $\Vert x_n-x\Vert_E\leq 4\Vert a_n-a\Vert_E+8r\Vert L_n-L\Vert_{E\to E}.$
    \end{proof}
\end{lemma}
\section{Generalized Singular Fractional Burgers-type equation}\label{sec;GSFB}
\par In this section, we introduce a Generalized  Singular Fractional Burgers-type equation. 
\subsection{Triple \texorpdfstring{$(V,\ml T,b)$}{tree}}
We first introduce an index set $\ml T$ that describes the algebraic structure of the singular expansion.
\begin{definition}\label{def;regular sequence}
Let
\[
\ml T(\tau^*):=
\{\tau_1\cdot\tau_2:\tau_1,\tau_2\in\ml T\}
\cup\{\tau^*\},
\]
where $\ml T(\tau^*)$ is equipped with a commutative multiplication operation $\cdot$.

\end{definition}

For each $\tau\in\ml T$, we associate a stochastic process $X^\tau$. To characterize the regularity of these processes, we introduce a family of Banach spaces
$\{V_\eta\}_{\eta\in\mathbb R}$ and establish a correspondence between the index set $\ml T$ and the regularity scale. In particular, we assume that
\[
X^{\tau^*}\in V_\eta .
\]
To describe the regularity of the higher-order objects $X^\tau$, we introduce a function
\[
r:\ml T\rightarrow\mathbb R, r(\tau^*)=\alpha,
\]
which satisfies the recursive relation
\begin{equation}\label{eq;r}
    r(\tau_1\cdot\tau_2)
=
\min
\big(
r(\tau_1),
r(\tau_2),
r(\tau_1)+r(\tau_2)
\big)+b,
\end{equation}
namely, $b$--order regular function on $\ml T$. Here $b>0$ represents the smoothing effect arising from the expansion and plays an essential role in ensuring that the recursive construction is well defined.It is not difficult to observe the following two facts.
\begin{proposition}\label{prop;tau}
Let $b>0$, then 
    \begin{itemize}
        \item (1) For any $\tau\in\ml T$, we have $r(\tau)$ is determined by $r(\tau^*)=\alpha$. 
        \item (2) For any $\tau\in\ml T$ and $b+r(\tau^*)>0$, we have $r(\tau)\geq r(\tau^*)$, the equality is hold if and only if $\tau=\tau^*$.
    \end{itemize}
\end{proposition}
We summarize above definition together
\begin{definition}
	Define a regular triple $(V,\ml T,r)$, where $\ml T$ satisfied Definition \ref{def;regular sequence} and $b>0$, let $r$ be a $b$-order regular function on $\ml T$. Let $V:\eta\mapsto V_\eta$ where $\{V_\eta\}_{\eta\in\m R}$ is a sequence of Banach space satisfied for $\eta\leq \bar\eta$, $V_{\bar\eta}\subset V_\eta$. Then we say that $X$ is adapted to the triple $(V,\ml T,b)$, if following two properties is satisfied 
	 \begin{itemize}
		\item[(1)]{$X:\ml T\mapsto \cup_{\alpha}V_\alpha$.}
		\item[(2)]{$X^\tau:=X(\tau)\in V_{r(\tau)}$.}
	\end{itemize}
\end{definition}
\subsection{Enhanced data}
 Let $c:\ml T\to \m N_+$ such that $c(\tau^*)=1$ and $c(\tau)=\sum_{\tau_1,\tau_2\in\ml T}\textbf{1}_{\tau_1\cdot\tau_2=\tau}c(\tau_1)c(\tau_2)$,  we firstly define a ``closed'' regular subset of $\ml T$ associated with regular function $r_b$.
 \begin{definition}
     Let $r_{\alpha,b}$ be a $b$-order regular function with $r_{\alpha,b}(\tau^*)=\alpha$. We say that $\ml T'(\tau^*)$ is a regular subset of $\ml T(\tau^*)$  associated with regular function $r_{\alpha,b}$  if it satisfies the following properties:
\begin{itemize}
    \item {$\tau^*\in\ml T'(\tau^*)$.}
    \item {There exists an ordering on $\ml T'(\tau^*)$ such that every $\tau_k\in\ml T'(\tau^*)\backslash \{\tau^*\}$ can be represented as
    \[
    \tau_k=\tau_1\cdot\tau_2,
    \]
    where $\tau_1,\tau_2$ belong to $\{\tau_i\}_{i<k}$.}
    \item {If $\tau_1,\tau_2\in\ml T'$ and $r_{\alpha,b}(\tau_1)+r_{\alpha,b}(\tau_2)\leq0$, then we have $\tau_1\cdot\tau_2\in\ml T'$}
\end{itemize}
 \end{definition}
 Assume that $\ml T'$ is a regular subset of $\ml T$, $u=\sum_{\tau\in \ml T'}c(\tau)X^\tau+v$, then the nonlinearity $f(D)(u^2)$ can be expressed as 
\begin{align*}
   f(D)(\sum_{\tau,\tau'\in \ml T'}c(\tau)c(\tau')X^\tau X^{\tau'}+2v\sum_{\tau\in \ml T'}c(\tau)X^\tau+v^2).
\end{align*}
We firstly assume that $v\cdot\sum_{\tau\in \ml T'}c(\tau)X^\tau$ is always well-defined. With the aid of the function  $r$, we split the set  $(\tau,\tau')\in\ml T'\times\ml T'$ into two parts, one consists of those pairs satisfying  $r(\tau)+r(\tau')\geq0$ which we denote by
\begin{equation}
    \ml R=\{(\tau_1,\tau_2)\::\tau_1,\tau_2\in\ml T'(\tau^*),r(\tau)+r(\tau')>0\},
\end{equation}
so that $\sum_{\tau,\tau'\in\ml T'}=\sum_{(\tau,\tau')\in \ml R}+\sum_{(\tau,\tau')\in\ml T'\times \ml T'\backslash\ml R}$. For the remaining singular interactions $(\tau,\tau')\in \ml T'\times \ml T'\backslash\ml R$, the products cannot be defined in the classical sense. We therefore introduce additional enhanced objects through a heat-flow regularization procedure, which provides a renormalized interpretation of these singular products.
\begin{equation}\label{eq;formal calculation}
   \ml LX^{\tau\cdot\tau'}=f(D)(X^{\tau}\cdot X^{\tau'}), \quad \ml Lv=\ml Lu-\sum_{\tau\in\ml T'}c(\tau)\ml LX^{\tau}
\end{equation}
Thus, in order to well-define the equation for $v$, we need to define the set  $\m X=\{X^\tau\}_{\tau\in \ml T'}$  which we refer to  a collection of enhanced objects containing all information needed to interpret the singular interactions. 
\begin{definition}
     We define a enhance data $\m X$ is generated by $\ml T'\subset \ml T$ if following three properties is satisfied
	 \begin{itemize}
		\item[(1)]{$\m X=\{X^{\tau}\}_{\tau\in\ml T'}$  where $X$ is adapted to triple $(V,\ml T,r)$.}
		\item[(2)]{$X^{\tau^*}\in\m X$.}
		\item[(3)]{If $\tau_1,\tau_2\in\ml T'$ satisfied $r(\tau_1)+r(\tau_2)\leq0$, we have $X^{\tau_1\cdot\tau_2}\in\m X$.  }
		Since every thing in $\m X$ is belong to some Banach space, by Proposition \ref{prop;tau} (1), assume $\ml T'$ is finite, we can use a uniform notation to define the $X^{\alpha,b}$ norm for $\m X\in \prod_{\tau\in\ml T'}V_{r(\tau)}$ as 
		\begin{equation}
			\Vert \m X\Vert_{X^{\alpha,b}}=\sum_{\tau\in\ml T'} \{\Vert X^{\tau}\Vert_{V_{r(\tau)}}:r_b(\tau^*)=\alpha\},
		\end{equation}
		where $r$ be a $b$-order regular function.
	\end{itemize}
\end{definition}
\begin{remark}
The condition $2\alpha+b>0$ guarantees that the enhanced data $\m X$ contains only finitely many singular components and that the mixed product
\[
v\cdot \sum_{\tau\in\ml T'}c(\tau)X^\tau
\]
is well-defined. When $2\alpha+b<0$, the enhanced data may become infinite and the above product requires further renormalized interpretations. These more singular cases are beyond the scope of the present work; we refer to 
\cite{gubinelli2015paracontrolled,hairer2013solving,hairerTheoryRegularityStructures2014}
for related theories based on paracontrolled calculus and regularity structures.
\end{remark}
We now reformulate the original problem as finding the regular component $v$ associated with an enhanced data $\m X$ satisfying suitable assumptions. More precisely, for a general enhanced data $\m X$, we consider the Generalized Singular Fractional Burgers-type(GSFB) equation
\begin{equation}\label{eq;gRBE;subcritcal}
\begin{aligned}
u&=\sum_{\tau\in\ml T'}c(\tau)X^\tau+v,\\
\partial_t v-\Lambda^\gamma v
&=f(D)(v^2)+f(D)\left(\sum_{(\tau,\tau')\in\ml R}c(\tau)c(\tau')X^\tau X^{\tau'}\right)+2f(D)\left(v\sum_{\tau\in\ml T'}c(\tau)X^\tau
\right).
\end{aligned}
\end{equation}
The first step is to establish the well-posedness of the above equation for
the regular component $v$ under general assumptions on the enhanced data
$\m X$. Let $T\geq0$, define $V_{r(\tau)}(T)=C_TW^{r(\tau)}=C_T(H^{r(\tau)}\cap\ml C^{r(\tau)})$,we introduce the norm of $\m X$ by
\begin{align*}
\|\m X\|_{X_T^{\alpha,b}}
&=
\sum_{X^\tau\in\m X}
\|X^\tau\|_{C_T(\ml C^{r(\tau)}\cap H^{r(\tau)})}\\
&=
\sum_{X^\tau\in\m X}
\|X^\tau\|_{C_TW^{r(\tau)}} .
\end{align*}
Under this framework, we obtain the following theorem.
\begin{theorem}\label{th;local result;gRBE;subcritical}
	 Let $\alpha<0$, $\gamma> 1$ and $0<b\leq\gamma-1$ satisfied 
	 $$2\alpha+b>0.$$ Let $f(D)=g(D)D$, where $g$ satisfies Assumption \ref{ass:g}. Given $\m X$ be a $b$-order enhance data generated by a regular set $\ml T'$ satisfied $(\m X,u_0)\in X_T^{\alpha,b}\times W^{s}$, then we have
	 \begin{itemize}
		\item{ (Existence and uniqueness) If  $-\alpha<s<\rho=\alpha+b$, then there exist $T^*$ and a unique $(u,v)\in C_{T^*}W^\alpha\times C_{T^*}W^s$ such the equation \eqref{eq;gRBE;subcritcal} is satisfied on $[0,T^*]$ for $u(0)=u_0+\sum_{\tau\in\ml T'}c(\tau) X^\tau(0)$. We call that $(u,v)$ is the solution of  equation \eqref{eq;gRBE;subcritcal} driven by $(\m X,u_0)$ on $[0,T]$.}
		\item{(Blow up time)The maximal time of existence $\bar T$ of such a solution does not dependent on $s$ and satisfied
		\begin{equation}
			\bar T=\sup\{t>0; \Vert v\Vert_{C_{t}W^s}<\infty\}.
		\end{equation}
		}
		\item{(Continuous dependence)
	Assume that  $\Vert u_0^n-u_0\Vert_{W^s}\to0$ and  $\Vert \m X^n-\m X\Vert _{X_{T}^{\alpha,b}}\to0$ for $T>0$, $T^*\leq T$ and $(u,v)$ be the solution of \eqref{eq;gRBE;subcritcal}  driven by $(\m X,u_0)$ on $[0,T^*]$, then there exists $T_n$ satisfying $\lim_{n\to\infty} T_n= T^*$ and for any $n\in \m N$, we have $(u^n,v^n)$ is the solution of \eqref{eq;gRBE;subcritcal} driven by $(\m X^n,u^n_0)$ on $[0,T_n]$. Moreover, we have $\Vert u_n-u\Vert_{C_{T_n\wedge T^*}W^\alpha}\to0$ as $\sigma\to0$. }
	\end{itemize}
\end{theorem}
Finally, we give a useful figure to 
\begin{figure}[H]
	\centering
	\begin{tikzpicture}
		\matrix (m) [matrix of math nodes,row sep=6em,column sep=8em,minimum width=6em]
		{   \eqref{eq;gRBE;subcritcal} & {(\tilde u,\tilde v)}\\
			\eqref{eq;generalized Burgers-type} & u \\
			\eqref{eq;Burgers;app}& u^\epsilon \\};
		\path[-stealth]
		(m-2-1) edge node [left] {Approximated} (m-3-1)
		(m-3-1.east|-m-3-2) edge node [below] {solution}(m-3-2)
		(m-1-1) edge node [midway,below,sloped]{Theorem \ref*{th;local result;gRBE;subcritical}} (m-1-2)
		(m-1-2) edge node [right]{Enhance data} (m-2-2) 
		(m-2-2) edge  node [midway] {×} (m-2-1)
		(m-2-1) edge node [left] {Generalized} (m-1-1)
		(m-3-2) edge node [right] {Theorem \ref*{th;main result}}  (m-2-2);
	\end{tikzpicture}
\end{figure}

\section{Local well-posedness}\label{sec;well-posedness}
\par Before we prove Theorem \ref{th;local result;gRBE;subcritical}, we firstly provide following two useful lemmas
\begin{lemma}\label{le;r;1}
	Let $r$ be $b$-order regular function on $\ml T'$ satisfied $r(\tau^*)=\alpha<0$, $\alpha+b>0$. Then for any $\tau\in \ml T'\backslash\{\tau^*\}$, we have $r(\tau)\geq 2\alpha +b$.
	\begin{proof} 
            We first note that  
            \[
r(\tau^*\cdot\tau^*)=
\min(\alpha,\alpha,2\alpha)+b
=2\alpha+b .
\]
Hence the desired inequality holds for $\tau=\tau^*\cdot \tau^*$. Suppose, on the contrary, that there exists some $\tau\in\ml T'$ such that
\[
r(\tau)<2\alpha+b.
\]
Since $\tau\neq \tau^*,\tau^*\cdot\tau^*$, by the construction of $\ml T'$, there exists $\tau_1,\tau_2\in\ml T'$ (namely, subtree) such that $\tau=(\tau_1\cdot\tau_2)$. By the definition of  $b$-order regular function $r$,
		\begin{align*}
			r(\tau)=\min(r(\tau_1),r(\tau_2),r(\tau_1)+r(\tau_2))+b<2\alpha+b.
		\end{align*}
		We claim that at least one of $r(\tau_1)$ and $r(\tau_2)$ is less than $2\alpha+b$, namely
        \[
        r(\tau_1) <2\alpha+b\quad  or\quad r(\tau_2) <2\alpha+b
        \]
       Otherwise, assume that 
       \[
       r(\tau_1)\geq2\alpha+b\quad and\quad r(\tau_2)\geq2\alpha+b,
       \]
       and consequently
		\begin{align*}
			r(\tau)\geq\min(2\alpha+b,4\alpha+2b)+b=\min (2\alpha+2b,4\alpha+3b)
		\end{align*}
		Since $\alpha+b>0$, we have $2\alpha+2b>2\alpha+b$ and $4\alpha+3b=2\alpha+b+2(\alpha+b)>2\alpha+b$,which contradicts the assumption, therefore the claim follows. Hence, for the given $\tau$, we can choose a subtree $\tau_1$ or $\tau_2$ denoted by $\tau'$, such that
        \[
        r(\tau')<2\alpha+b.
        \]
        Repeating the above argument, we obtain a descending finite sequence
		\begin{align*}
			\tau^0\to\tau^1\to\cdots \tau^n\to \cdots, \quad \tau^m\in\ml T',  m\in\m N.
		\end{align*}
		
       where $\tau^0=\tau$, and each $r(\tau^m)<2\alpha+b$, and $\tau^{m+1}$ is a subtree of $\tau^{m+1}$ . The above procedure can only be implemented a finite number of times, hence there exists a longest chain 
		\begin{align*}
			\tau^0\to\tau^1\to\cdots \tau^N.
		\end{align*}
	with$\tau^N=\tau^*$ and $\tau^{N-1}=\tau^*\cdot \tau^*$. However $r(\tau^*\cdot \tau^*)=2\alpha+b$, we have reached a contradiction which prove our lemma.
	\end{proof}
\end{lemma}
\begin{lemma}
   Let $r$ be a $b$-order regular function with $r(\tau^*)=\alpha$. $\ml T'(\tau^*)$ is a regular subset of $\ml T(\tau^*)$  associated with regular function $r$, if  $2\alpha+b >0$, then there exists a $\tau\in \ml T'$ such that $(\tau,\tau^*)\in \ml R$. 
    \begin{proof}
     If $\alpha\geq0$, then the lemma is obviously hold by choosing $\tau=\tau^*$. Now assume that $\alpha<0$, and $\bar \tau\in\ml T'$ satisfying $r(\bar \tau)+r(\tau^*)<0$, then by the definition of $\ml T'$, we have $\bar\tau\cdot \tau^*\in\ml T'$. Noting that $\tau=\bar\tau\cdot \tau^*\neq \tau^*$, it then follows from the definition of regular function $r$ and Lemma \ref{le;r;1} that
     \begin{align*}
r(\tau^*)+r(\bar\tau\cdot\tau^*)=&r(\tau^*)+\min(r(\tau^*),r(\bar\tau),r(\bar\tau)+r(\tau^*))+b\\
           =&2r(\tau^*)+b=2\alpha+b>0.
       \end{align*}
    Then we have $(\tau,\tau^*)\in \ml R$, which finish our proof.
       \end{proof}
\end{lemma}

\begin{proof}[proof of Theorem \ref{th;local result;gRBE;subcritical}]
	Fix $M:=C\max(\Vert u_0\Vert_{W^s},\Vert \m X\Vert_{X_T^\alpha})$ where $C$ is the constant in Lemma \ref{le;est;heat flow}, $r:=2M$ and $\delta=\alpha+b-s>0$. Rewrite equation \eqref{eq;gRBE;subcritcal} in mild form $v:=a+L(v)+\ml B(v,v)$, where 
	\begin{align*}
		\ml La=f(D)(\sum_{(\tau,\tau')\in\ml R}c(\tau)c(\tau')X^{\tau}X^{\tau'}),\quad &a(0)=u_0,\\
		\ml L(L(v))=2f(D)(v\sum_{\tau\in \ml T'}c(\tau)X^{\tau}),\quad &L(v)(0)=0,\\
		\ml L(\ml B(v,v))=f(D)(v^2),\quad &B_f(v,v)(0)=0.
	\end{align*}
	  By Lemma \ref{le;r;1}, for $\tau\in\ml T\backslash\tau^*$, $r(\tau)\geq 2\alpha+b$, then we have $r(\tau)\geq\alpha$ since $\alpha+b>0$. By definition of $\ml R$, for any $(\tau,\tau')\in\ml R$, we have $r(\tau)+r(\tau')>0$, then there at least one of the two factors has positive regularity
	\begin{align*}
		r(\tau)>0\quad or\quad r(\tau')>0.
	\end{align*}
	Then by Lemma \ref{le;est;paraproduct}, we have
	\begin{equation}\label{eq;R;1}
		\begin{aligned}
			\Vert \sum_{(\tau,\tau')\in\ml R}X^{\tau}X^{\tau'}\Vert_{C_TW^\alpha}\lesssim\Vert X^{\tau}\Vert_{C_TW^{r(\tau)}}\Vert X^{\tau'}\Vert_{C_TW^{r(\tau')}}\lesssim \Vert \m X\Vert_{X_T^{\alpha,b}}^2.
		\end{aligned}
	\end{equation}
	 Since $0<s<\rho$, then by Lemma \ref{le;est;heat flow} and \eqref{eq;R;1}, for any $T\leq 1\wedge (\frac{1}{4CM})^{\frac{\gamma}{\delta}}$, we have
	\begin{align*}
		\Vert a\Vert_{C_{T}W^{s}}\leq& \frac{r}{2}+CT^{\frac{\delta}{\gamma}} \Vert \ml La\Vert_{C_{T}W^{s-\gamma+\delta}}\\
		\leq&\frac{r}{2}+CT^{\frac{\delta}{\gamma}} \Vert \sum_{(\tau,\tau')\in\ml R}X^{\tau}X^{\tau'}\Vert_{C_TW^\alpha}\\
		\leq&\frac{r}{2}+CT^{\frac{\delta}{\gamma}} \Vert \m X\Vert_{X_T^{\alpha,b}}^2<r.
	\end{align*}
     Noting that $X^{\tau^*}\in C_TW^{\alpha}$ and $s+\alpha>0$,  by Lemma \ref{le;est;heat flow} and Lemma \ref{le;est;paraproduct}, for any $T\leq 1\wedge (\frac{1}{8CM})^{\frac{\gamma}{\delta}}$,  we have
     \begin{align*}
     	\Vert L(v)\Vert_{C_{T}W^{s}}\leq& CT^{\frac{\delta}{\gamma}}\Vert f(D)(v\sum_{\tau\in \ml T}X^{\tau})\Vert_{C_{T}W^{\alpha-1}}\\
     	\leq& CT^{\frac{\delta}{\gamma}}\sum_{\tau\in \ml T}\Vert vX^{\tau}\Vert_{C_{T}W^{\alpha}}\\
     	\leq &CT^{\frac{\delta}{\gamma}}\Vert v\Vert_{C_{T}W^{s}}\sum_{\tau\in \ml T}\Vert X^{\tau}\Vert_{C_{T}W^{\alpha}}\\
     	\leq &CT^{\frac{\delta}{\gamma}}\Vert v\Vert_{C_{T}W^{s}}\Vert \m X\Vert_{X^{\alpha}_{1}}\\
     	<&\frac{1}{4}\Vert v\Vert_{C_{T}W^{s}}.
     \end{align*}  
    Finally, for the bilinear term, since $\gamma>1$, for any $T\leq 1\wedge(\frac{1}{32CM})^{\frac{\gamma}{\gamma-1}}$ such that
     \begin{align*}
     	\Vert \ml B_f(v,v)\Vert_{C_{T}W^{s}}\leq& CT^{\frac{\gamma-1}{\gamma}}	\Vert v\Vert_{C_{T}W^{s}}^2.\\
     	<&\frac{1}{8r}\Vert v\Vert_{C_{T}W^{s}}^2.
     \end{align*}
    Combining the above estimates, all assumptions of Lemma  \ref{le;Picard th} are satisfied for $E=C_{T^*}W^s$ with $T^*\leq T$. Therefore, there exists a unique $v\in C_{T^*}W^s$ satisfied \eqref{eq;gRBE;subcritcal} where $T^*$ depends only on $M$. 
    \par Secondly, we obtain the solution until blow up time. Assume that we constructed a solution $u$ on $[0,h]$ for some $h>0$ and consider the time interval $[h,h+1]$. Denote the solution on this interval by $\tilde u$, and write $\tilde{\m X}(t)=\m X(t+h)$ which means $\tilde X^{\tau}(t)=X^{\tau}(t+h)$ for any $X^{\tau}\in \m X$. The initial condition for $\tilde u$ is $\tilde u(0)=u(h)\in W^\alpha$, it's not the same setting as before. But we only needed a smooth initial condition of regularity $s$ for the remainder, note that 
    \begin{align*}
    	\tilde v(0)=&\tilde u(0)-\sum_{\tau\in\ml T'}c(\tau) \tilde X^{\tau}(0)\\
    	=&u(h)-\sum_{\tau\in\ml T'}c(\tau) X^{\tau}(h)\\
    	=&v(h)\in W^{s}.
    \end{align*}
    As a consequence, let $\tilde{M}:=C\max(\Vert v(h)\Vert_{W^s}, \Vert \m X\Vert_{X_{h+1}^{\alpha,b}})$, then there exists a unique solution $(u,v)$ of \eqref{eq;gRBE;subcritcal} on $[0,h(\tilde M)]$ driven by $(v(h),\m X(t+h))$. The second result then follows by choosing $h>0$ arbitrarily such that  $h<\bar T$ .
    \par Finally, we prove the continuous dependence of the solution with respect to $(\m X,u_0)$ up to blow up time. Since $u_0^n\in W^s$ and $\m X\in X_T^{\alpha,b}$, the first result implies that for every $n\in\m N$, there exists a positive time $T_n$ and unique solution $(u^n,v^n)$ of\ \eqref{eq;gRBE;subcritcal} driven by $(\m X^n,u_0^n)$ on $[0,T_n]$. It then remain to show the convergence. 
    \par Let $0<T^*< \bar T\wedge T$ , and let $(u,v)$ denote the solution of \eqref{eq;gRBE;subcritcal} driven by $(\m X,u_0)$ on $[0,T^*]$. Since $u_0^{n}\to u_0$ in $W^s$, $\m X^n\to \m X$ in $X_T^{\alpha,b}$, for any sufficiently small $\delta>0$ and $\epsilon_0>0$, there exists $N(\delta,\epsilon_0)$ such that for all $n>N(\delta,\epsilon_0)$
    \begin{align*}
        \Vert u_0^n-u_0\Vert_{W^s}<\epsilon_0,\quad \Vert \m X^n-\m X\Vert_{X_T}<\delta.
    \end{align*}
    Let $M=(1\vee C)\max (\Vert v\Vert_{C_TW^s},\Vert \m X\Vert_{X_T^\alpha})$, and $r=2M+2$, by Lemma \ref{le;Picard th}, there exists $N\in \m N$ such that for any $n> N\vee N(\delta,\epsilon_0)$
\begin{align*}
    \Vert v^{n}(t)-v(t)\Vert_{C_{T(M)}W^s}\leq&  4\Vert a(t)-a_n(t)\Vert_{W^s}+8r\Vert L_n-L\Vert_{W^s\to W^s}\\
    \leq&C(\epsilon_0+\delta),
\end{align*}
    where $C$ is only dependent on $M,T^*$. Now choosing $k=1,2,\cdots,[T/T(M)]+1$, and $\epsilon,\epsilon_{k}\leq 1$ satisfying
    \begin{equation}
      (1\vee C)(\epsilon_k+\delta)\leq\epsilon_{k+1},\quad \epsilon_{[T/T(M)]+1}=\epsilon.
    \end{equation}
  Define $(u^{n,k},v^{n,k})$ to be the solution of \eqref{eq;gRBE;subcritcal} driven by $(\m X^{n,k},v^{n,k-1}(T(M)))$ on $[0,T(M)]$ where $v^{n,0}=u^n_0$ and 
  \begin{align*}
      \m X^{n,k}=\{X^{n,\tau}(t+k(T(M)))\}_{\tau\in\ml T'}.
  \end{align*}
  Particularly, $(u^{n,0},v^{n,0})=(u^n,v^n)$ on $[0,T(M)]$. We claim that for any $k$, 
\begin{equation}\label{eq;induced claim}
    \Vert v^{n,k}\Vert_{C_{T(M)}W^s}\leq M+1,\quad \Vert v^{n,k}(\cdot)-v(\cdot+kT(M))\Vert_{C_{T(M)}W^s}\leq \epsilon_{k+1}.
\end{equation}
 Indeed, the case $k=0$ follows immediately from the previous estimates. Assume \eqref{eq;induced claim}  is hold for $k=m$, then $\Vert v^{n,m+1}(0)\Vert_{W^{s}}\leq M+1$, which ensure that $(u^{n,m+1},v^{n,m+1})$ exists on $[0,T(M)]$. To estimate $\Vert v^{n,m+1}(\cdot)-v(\cdot+(m+1)T(M))\Vert_{C_{T(M)}W^s}$, noting that 
 \begin{align*}
     \Vert v^{n,m+1}(0)-v((m+1)T(M))\Vert_{W^s}= \Vert v^{n,m}(T(M))-v(T(M)+mT(M))\Vert_{W^s}\leq \epsilon_{m+1}<\epsilon,
 \end{align*}
 then similar as the case of $k=0$, we have
 \begin{align*}
      \Vert v^{n,m+1}(\cdot)-v(\cdot +(m+1)T(M))\Vert_{C_{T(M)}W^s}\leq &4\Vert a^{n,m+1}-a\Vert_{C_{T(M)}W^s} +8\Vert L^{n,m+1}-L\Vert_{C_{T(M)}W^s\to C_{T(M)}W^s} \\
 \leq&C(\epsilon_{m+1}+\delta)\leq \epsilon_{m+2}\leq 1.
 \end{align*}
then $\Vert v^{n,m+1}\Vert_{C_{T(M)}W^s} \leq M+1$, and the induced claim is hold. Noting that $(u^{n},v^{n})=(u^{n,k},v^{n,k})$ on $[(k-1)T(M),T\wedge kT(M)]$ is the solution of \eqref{eq;gRBE;subcritcal} driven by $(\m X_n, u^n_0)$, by uniqueness of solutions, the above induction argument yields that for any fixed $\epsilon>0$,  we have 
 \begin{equation}
     \Vert v^{n}-v\Vert_{C_{T\wedge T^*}W^s}\leq \max (\epsilon_k)\leq \epsilon,
 \end{equation}
 for $n>N(\epsilon,\delta)$. Therefore after redefining $T_n= T$ for $n>N(\epsilon,\delta)$ while leaving the finitely many remaining $T_n$'s unchanged , we conclude that
 \begin{align*}
     \lim_{n\to \infty}\Vert v^n-v\Vert_{C_{T_n\wedge T^*}W^s}=0.
 \end{align*}
Finally, choose a sequence \(T_m\uparrow\bar T\). For each fixed \(T_m\), the above argument provides a sequence \(T_{n,m}\) such that $T_{n,m}\to T_m$ and
\[
\lim_{n\to\infty}
\|v^n-v\|_{C([0,T_{n,m}\wedge T_m];W^s)}
=0 .
\]
By a diagonal argument, we can choose a sequence \(T_n\to\bar T\) such that
\[
\lim_{n\to\infty}
\|v^n-v\|_{C([0,T_n\wedge\bar T);W^s)}
=0 .
\]
This completes the proof.
\end{proof}
Next, we give energy estimate towards Burgers equation
\begin{proposition}
    Let  $\gamma>1$,  $2\alpha+b>0$, $\alpha>1-\frac{\gamma}{2}$, and $\m X\in X_T^{\alpha,b} $. Let $f(D)=D$.$(u,v)$ is the solution of \eqref{eq;gRBE;subcritcal} on $[0,T]$ driven by $(u_0,\m X)$, then for any $t\in[0,T]$, we have the estimate
\begin{align*}
    \Vert v(t)\Vert_{L^2}^2+2\int_0^t\Vert v(t)\Vert_{H^{\frac{\gamma}{2}}}^2\mathrm ds\leq& \Vert u_0\Vert^2_{L^2}\exp({C_\nu(1+\Vert \m X\Vert_{X_T^{\alpha,b}}^{\frac{4\gamma}{\gamma-1}})t})\\
    &+C_\nu\Vert \m X\Vert_{C_{ T}W^\alpha}^4\int_0^t e^{C_\nu(1+\Vert \m X\Vert_{X_T^{\alpha,b}}^{\frac{4\gamma}{\gamma-1}})(t-t')} \mathrm{d}t'.
\end{align*}
    \begin{proof}
        Observing that $v$ satisfies \eqref{eq;gRBE;subcritcal}. Multiplying $v$ on the both side of \eqref{eq;gRBE;subcritcal} and integrating by parts, we have
			\begin{align*}
				\Vert v(t)\Vert_{L^2}^2+2\int_0^t \Vert v(t')\Vert_{H^\frac{\gamma}{2}}^2\mathrm{d}t'\leq& \Vert u_0\Vert_{L^2}^2+2\int_0^t\Vert v(t')\Vert_{L^2}^2\mathrm{d}t'+2\int_0^t\langle v(t'),2\p_x(v(t')\sum_{(\tau,\tau')\in\ml R}c(\tau)X^\tau)\rangle\mathrm{d}t'\\
                &+\int_0^t\langle  v(t'),\p_x (\sum c(\tau)c(\tau')X^\tau X^{\tau'})\rangle\mathrm{d}t'.
			\end{align*}
		By \eqref{eq;R;1}, we have  
        \begin{equation}
            \sum_{(\tau,\tau')\in\ml R}\Vert X^\tau X^{\tau'}\Vert_{C_TW^\alpha} \leq \Vert \m X\Vert_{X_T^{\alpha,b}}^2.
        \end{equation}
        \
            By H\"older inequality, Young's inequality, Lemma \ref{le;est;paraproduct},  for $\alpha>1-\frac{\gamma}{2}$, we have
			\begin{align*}
				\langle  v(t'),\p_x (\sum_{(\tau,\tau')\in\ml R} c(\tau)c(\tau')X^\tau X^{\tau'})\rangle\leq& C\Vert v(t')\Vert_{H^{\frac\gamma2}}\sum_{(\tau,\tau')\in\ml R}\Vert X^{\tau}X^{\tau'}\Vert _{H^{1-\frac\gamma 2}}\\
                \leq&\nu\Vert v(t)\Vert_{H^{\frac\gamma2}}^2+C_\nu\Vert \m X\Vert_{X_{t'}^{\alpha,b}}^4
			\end{align*}
         Similarly, we have
         \begin{align*}
             \langle v(t'),2\p_x(v(t')\sum_{(\tau,\tau')\in\ml R} c(\tau)X^\tau)\rangle\leq& C\Vert v(t')\Vert_{H^{\frac\gamma2}}\Vert v(t')\sum_{(\tau,\tau')\in\ml R} c(\tau)X^\tau\Vert_{H^{1-\frac\gamma 2}}\\
             \leq& C\Vert v(t')\Vert_{H^{\frac\gamma2}}\Vert v(t')\Vert_{H^{\frac\gamma4+\frac
             14}}\Vert \m X\Vert_{X_{t'}^\alpha}\\
             \leq&\frac{\nu}{2} \Vert v(t')\Vert_{H^{\frac\gamma2}}^2+C_\nu\Vert \m X\Vert_{X_{t'}^\alpha}^{\frac{4\gamma}{\gamma-1}}\Vert v\Vert_{L^2}^2
         \end{align*}
			For $0\leq t\leq T$, we have 
			\begin{equation}
				\begin{aligned}
					\Vert v(t)\Vert_{L^2}^2+2\int_0^t \Vert v(t')\Vert_{H^\frac{\gamma}{2}}^2\mathrm{d}t'\leq& \Vert u_0\Vert_{L^2}^2+\nu \int_0^t \Vert v(t')\Vert_{H^\frac{\gamma}{2}}^2\mathrm{d}t'+C_\nu\int_0^t (1+\Vert \m X\Vert_{X_T^{\alpha,b}}^{\frac{4\gamma}{\gamma-1}})\Vert v(t')\Vert_{L^2}^2\mathrm{d}t'\\
					&+C_\nu\int_0^t\Vert \m X\Vert_{X_T^{\alpha,b}}^4\mathrm{d}t'.	
				\end{aligned}
			\end{equation}
			Observing that $\Vert\tilde X\Vert_{C_tW^\alpha}\leq \Vert \tilde X\Vert_{C_{\bar T}W^\alpha}$, for $\nu<1$, by Gronwall's inequality, the proposition is proved.
		\end{proof}
\end{proposition}
\begin{remark}\label{rem;energy}
When $\alpha<0$, the condition $\alpha>1-\gamma/2$ implies $\gamma>2-2\alpha>2$. Thus, for negative spatial regularity, this energy estimate requires stronger dissipation than the regime $\gamma\in(3/2,2]$ considered in the main theorem.
\end{remark}
	\section{Driving by \texorpdfstring{$|D|^{1/2}\xi$}{fractional noise}}\label{sec;high order}
	In this section, we restrict $\gamma\in(\frac{3}{2},2]$ and $\sigma=\frac{1}{2}$. Let us begin with \eqref{eq;Burgers;app}. To define it completely, choosing a function $\phi:\mathbb{R}\mapsto\mathbb{R}_+$, that is even, smooth, compactly supported, decreasing on $\mathbb{R}_+$, and such that $\phi(0)=1$, and then setting 
	$$\xi_\epsilon(x)=\ml F^{-1}(\phi(\epsilon k)\ml F(\xi)(k))(x),\quad u^\epsilon_0(x)=\ml F^{-1}(\phi(\epsilon k)\ml F(u_0)(k))(x).$$
	Here $\ml F$ is spatial Fourier transform. Its obviously that $u^\epsilon_0(x),\xi_\epsilon(x)\in C^\infty(\m T)$ is smooth by distribution theory. The approximation is smooth in space, while the noise remains white in time. Let $Y_\epsilon$ be the stationary solution of linear equation
\begin{equation}\label{eq;linear evolution;noise;1/2}
	\p_t Y_\epsilon-\Lambda^\gamma Y_\epsilon=\vert D\vert^{\frac{1}{2}}\xi_\epsilon,\quad Y_\epsilon(0)=\int_{-\infty}^0e^{-s\Lambda^\gamma}|D|^{\frac{1}{2}}\xi_\epsilon(s)\mathrm ds,
\end{equation}
	 denoting $u^\epsilon=u^{\epsilon,1}+Y_\epsilon$, we have
	\begin{align*}
		\ml Lu^{\epsilon,1}=f(D)(Y_\epsilon^2+(u^{\epsilon,1})^2+2Y_\epsilon u^{\epsilon,1}),\quad u^{\epsilon,1}(0)=u_0^\epsilon-Y_\epsilon(0).
	\end{align*}
    It can be solved by similar argument in \cite{zhang2025resultsfractionalroughburgers}, where $u^{\epsilon,1}\in C_T\ml C^\infty$. However, we cannot obtain the solution of equation \eqref{eq;generalized Burgers-type} by taking the limit of $u^\epsilon$. Although \(Y_\epsilon\) converges, its negative spatial regularity prevents us from passing to the limit directly in the nonlinear term.
	\begin{proposition}\label{prop;order-1}
		Let $\gamma>3/2$ and $\sigma=1/2$, then for any $T>0$, there exists some $\alpha>-1/4$ such that  such that $Y_\epsilon\to Y$ in $L^2(\Omega; C_TW^{\alpha})$. 
	\end{proposition}
	The main difficult is that $Y$ only has $-\frac{1}{4}+$ order regularity(H\"older or Sobolev), $Y\cdot Y$ is ill-defined. Define 
    \begin{equation}
        B_f(X,Y)=\int_0^t P(t-s)f(D)(X\cdot Y)\mathrm ds.
    \end{equation}
Let $X_\epsilon^{\RS {lr}}=B_f(Y_\epsilon,Y_\epsilon)$ and $u^{\epsilon,1}=X_\epsilon^{\RS {lr}}+u^{\epsilon,2}$, we have $u^{\epsilon,2}$ satisfied
	\begin{align*}
		\ml Lu^{\epsilon,2}=f(D)((u^{\epsilon,2})^2+2u^{\epsilon,2} X_\epsilon^{\RS {lr}}+(X_\epsilon^{\RS {lr}})^2+2Y_\epsilon X_\epsilon^{\RS {lr}}+2Y_\epsilon u^{\epsilon,2}).
	\end{align*}
	The convergence of $X_\epsilon^{\RS {lr}}$ is given by following proposition
	\begin{proposition}\label{prop;order-2}
			Let $\gamma>3/2$, $\sigma=1/2$ and $f(D)=g(D)D$ where $g$ satisfies Assumption \ref{ass:g}, then for any $T>0$, there exists some $\alpha>0$ and $X^{\RS {lr}}\in C_TW^\alpha$ such that $X_\epsilon^{\RS {lr}} \to X^{\RS{lr}}$ in $L^2(\Omega;C_TW^\alpha)$.
	\end{proposition}
Let $X_{\epsilon}^{\RS{rLlr}}=B_f(Y_\epsilon, B_f(Y_\epsilon,Y_\epsilon))$ and $u^{\epsilon,2}:=2 X_\epsilon^{\RS {rLlr}}+u^{\epsilon,3}$, then $u^{\epsilon,3}$ satisfied 
	\begin{equation}\label{eq;higher order expansion;3}
		\begin{aligned}
			\ml Lu^{\epsilon,3}=&f(D)((u^{\epsilon,3})^2+4u^{(3)}X_\epsilon^{\RS {rLlr}}+4(X_\epsilon^{\RS {rLlr}})^2+2u^{\epsilon,3} X_\epsilon^{\RS {lr}}\\
			&+4X_\epsilon^{\RS {rLlr}}X_\epsilon^{\RS {lr}}+(X_\epsilon^{\RS {lr}})^2+2Y_\epsilon u^{\epsilon,3}+4Y_\epsilon X_\epsilon^{\RS {rLlr}} ).
		\end{aligned}
	\end{equation}
	Also, The convergence of $X_\epsilon^{\RS {rLlr}}$ is given by following proposition
	\begin{proposition}\label{prop;order-3}
		Let $\gamma>3/2$, $\sigma=1/2$  and $f(D)=g(D)D$ where $g$ satisfies Assumption \ref{ass:g},  then for any $T>0$, there exists some $\alpha>\frac{1}{4}$ and $X^{\RS{rLlr}}\in C_TW^\alpha$, such that $X_\epsilon^{\RS {rLlr}} \to X^{\RS{rLlr}}$ in $L^2(\Omega;C_TW^\alpha)$.
	\end{proposition}
	    All products in the remainder equation are well defined under the regularities established above. The convergence of \(u^{\epsilon,3}\) then follows from the local well-posedness and stability estimates, provided that the remainder initial data converge in the required space. 
\subsection{Convergence of Gaussian trees}
In the latter part of this chapter, we will prove Theorem \ref{th;main result}. We reiterate that the essence of this result is to demonstrate that the subcritical GFSB \eqref{eq;gRBE;subcritcal}, under specific enhance data, can represent the limit of solutions to the approximating Equation \eqref{eq;Burgers;app}. Before our analysis, we introduce a notation to simplify the representation of regular set in Definition \ref{def;regular sequence}.
\par For our convenience, define $\tau^*=\RS{n}$ and $\overline {\ml T}$ is generated by $\tau^*$ as Definition \ref{def;regular sequence}, Moreover
\begin{align*}
	(\RS{n}\cdot \RS{n})=\RS{lr},\quad (\RS{n}\cdot \RS{lr})=(\RS{lr}\cdot \RS{n})=\RS{rLlr},\quad (\RS{rLlr}\cdot \RS{n})=\RS{rLrLlr},\quad (\RS{lr}\cdot \RS{lr})=\RS{{LL{lr}}{RR{lr}}}.
\end{align*}
and
\begin{align*}
	c(\RS{n})=0.\quad c(\RS{lr})=1,\quad c(\RS{rLlr})=2.
\end{align*}
The following Lemma is crucial in our argument 
\begin{lemma}\label{le;summary criterion}
	If $\alpha+\beta>1$, then $\sum_{k\neq \{a,0\}}\frac{1}{\vert k-a\vert^{\alpha}\vert k\vert^{\beta}}<\infty$ uniformly for $a\in \mathbb{Z}_*$.
	\begin{proof}
		Fix $R=\max(\vert a\vert,1)$. Define two sets $A=\{k:\vert k\vert\leq 2R,k\neq 0,a\}$ and $B=\{k:\vert k\vert>2R\}$. On the domain $B$, we have $\vert k-a\vert\geq \vert k\vert-\vert a\vert\geq \vert k\vert-R\geq \frac{\vert k\vert}{2}$, then the estimate
		\begin{align*}
			\sum_{k\in B}\frac{1}{\vert k-a\vert^{\alpha}\vert k\vert^{\beta}}\leq \sum_{k\neq 0}\frac{1}{(\vert k\vert/2)^{\alpha}\vert k\vert^{\beta}}<\infty.
		\end{align*}
		On the domain $A$,
		\begin{align*}
        \sum_{k\in A}\frac{1}{\vert k-a\vert^{\alpha}\vert k\vert^{\beta}}=&\sum_{k\in A,\vert k\vert\leq \vert k-a\vert}\frac{1}{\vert k-a\vert^{\alpha}\vert k\vert^{\beta}}+\sum_{k\in A,\vert k\vert> \vert k-a\vert}\frac{1}{\vert k-a\vert^{\alpha}\vert k\vert^{\beta}}\\
        \leq &\sum_{k\neq 0}\frac{1}{\vert k\vert^{1+}}+\sum_{k\neq {a}}\frac{1}{\vert k-a\vert^{1+}}\\
        \leq &\sum_{k\neq 0}\frac{1}{\vert k\vert^{1+}}+\sum_{k'\neq {0}}\frac{1}{\vert k'\vert^{1+}}.
		\end{align*}
		We have two sums, each possessing an upper bound independent of $a$, therefore the upper bound for the original sum is also independent of $a$. 
	\end{proof}
\end{lemma}
\begin{lemma}
Let $a_i\in\mathbb Z$ and $\alpha_i\geq0$ for
$1\leq i\leq n$. If $\sum_{i=1}^n\alpha_i>1$, then
\[
\sum_{k\in\mathbb Z\setminus\{a_1,\ldots,a_n\}}
\prod_{i=1}^n\frac{1}{|k-a_i|^{\alpha_i}}
\]
is bounded uniformly in $(a_1,\ldots,a_n)\in\mathbb Z^n$.
\begin{proof}
Partition $\mathbb Z\setminus\{a_1,\ldots,a_n\}$ into
sets $A_1,\ldots,A_n$ such that, for $k\in A_j$,
\[
|k-a_j|=\min_{1\leq i\leq n}|k-a_i|,
\]
assigning ties to the smallest index. Since $\alpha_i\geq0$,
\[
\begin{aligned}
\sum_{k\notin\{a_1,\ldots,a_n\}}
\prod_{i=1}^n |k-a_i|^{-\alpha_i}
&\leq
\sum_{j=1}^n\sum_{k\in A_j}
|k-a_j|^{-\sum_{i=1}^n\alpha_i}\\
&\leq
n\sum_{k\in\mathbb Z\setminus\{0\}}
|k|^{-\sum_{i=1}^n\alpha_i}
<\infty.
\end{aligned}
\]
The bound is independent of $a_1,\ldots,a_n$.
\end{proof}
\end{lemma}
\begin{lemma}\label{le;convolution decay}
Let $\alpha,\beta\geq0$ satisfy $\alpha+\beta>1$,
with $\alpha\neq1$ and $\beta\neq1$. Then
\[
\sum_{k\in\mathbb Z\setminus\{0,a\}}
\frac{1}{|k-a|^\alpha|k|^\beta}
\lesssim
|a|^{-\min\{\alpha,\beta,\alpha+\beta-1\}},
\qquad a\in\mathbb Z\setminus\{0\},
\]
where the implicit constant depends only on $\alpha,\beta$. Moreover, we have
\[
\sum_{k\in\mathbb Z\setminus\{0,a\}}
\frac{1}{(|k-a|+|k|)^{\alpha+\beta}}
\lesssim
|a|^{-\min\{\alpha,\beta,\alpha+\beta-1\}},
\qquad a\in\mathbb Z\setminus\{0\},
\]
\begin{proof}
Set $R=|a|\geq1$. We divide the summation into the regions
\[
0<|k|\leq R/2,\qquad
0<|k-a|\leq R/2,\qquad
|k|>R/2,\ |k-a|>R/2.
\]
Possible overlaps only increase the resulting upper bound.

On the first region, $|k-a|\geq R/2$, and hence
\[
\begin{aligned}
\sum_{0<|k|\leq R/2}
\frac{1}{|k-a|^\alpha|k|^\beta}
&\lesssim R^{-\alpha}
\sum_{0<|k|\leq R/2}|k|^{-\beta}\\
&\lesssim
\begin{cases}
R^{1-\alpha-\beta},&\beta<1,\\
R^{-\alpha},&\beta>1.
\end{cases}
\end{aligned}
\]
Similarly, on the second region,
\[
\sum_{0<|k-a|\leq R/2}
\frac{1}{|k-a|^\alpha|k|^\beta}
\lesssim
\begin{cases}
R^{1-\alpha-\beta},&\alpha<1,\\
R^{-\beta},&\alpha>1.
\end{cases}
\]
For the third region, first consider $|k|\leq2R$.
There are at most $4R+1$ such integers, and both
distances are greater than $R/2$. Therefore,
\[
\sum_{\substack{R/2<|k|\leq2R\\|k-a|>R/2}}
\frac{1}{|k-a|^\alpha|k|^\beta}
\lesssim R^{1-\alpha-\beta}.
\]
If $|k|>2R$, then $|k-a|\geq|k|/2$, so
\[
\sum_{|k|>2R}
\frac{1}{|k-a|^\alpha|k|^\beta}
\lesssim
\sum_{|k|>2R}|k|^{-\alpha-\beta}
\lesssim R^{1-\alpha-\beta}.
\]
Combining these estimates yields the first desired bound. Noting that $(|k-a|+|k|)^{\alpha+\beta}\geq |k-a|^\alpha|k|^{\beta}$ for $\alpha,\beta\geq0$, then the first bound yield the second bound. 
\end{proof}
\end{lemma}
\begin{corollary}\label{coro;summable}
Let $a,b,d,e\in\mathbb R$ and $c\geq0$ satisfy
\[
1+a+b<c+d+e<2+a+b,
\qquad e<1+b,\qquad d<1+a.
\]
Then
\[
\sum_{k\in\mathbb Z\setminus\{0,m\}}
\frac{|k-m|^a|k|^b}
{(|k-m|+|k|)^c|k-m|^d|k|^e}
\lesssim
|m|^{-(c+d+e-a-b-1)},
\qquad m\in\mathbb Z\setminus\{0\},
\]
where the implicit constant is independent of $m$.
\begin{proof}
Set $\delta=c+d+e-a-b-1$. By assumption,
$\delta\in(0,1)$, and
\begin{equation}\label{eq;corollary;endpoint}
c+d-a>\delta,
\qquad c+e-b>\delta.
\end{equation}
For every $t\in[0,1]$, we have
\[
(|k-m|+|k|)^c
\geq |k-m|^{(1-t)c}|k|^{tc}.
\]
Consequently,
\[
\frac{|k-m|^a|k|^b}
{(|k-m|+|k|)^c|k-m|^d|k|^e}
\leq
\frac{1}
{|k-m|^{d+(1-t)c-a}|k|^{e+tc-b}}.
\]

We claim that there exists $t\in[0,1]$ such that
\[
d+(1-t)c-a>\delta,
\qquad e+tc-b>\delta.
\]
Suppose otherwise. By \eqref{eq;corollary;endpoint},
at $t=0$ the first exponent is greater than $\delta$,
so the second is at most $\delta$. At $t=1$ the
second exponent is greater than $\delta$, so the first
is at most $\delta$. By continuity, there exists
$t_0\in(0,1)$ such that
\[
d+(1-t_0)c-a=e+t_0c-b.
\]
Under the supposition, both exponents must then be at
most $\delta$. However, their sum is
\[
c+d+e-a-b=1+\delta>2\delta,
\]
which is a contradiction.

For this choice of $t$, both exponents belong to
$(\delta,1)$, since their sum is $1+\delta$.
Applying Lemma \ref{le;convolution decay} therefore
gives the desired estimate.
\end{proof}
\end{corollary}

\subsection*{Construction of \texorpdfstring{$Y^{\RS{lr}}$ }}
We consider a more general setting. Let $\sigma>0$ and $Y_\epsilon$ be the solution to
\begin{equation}\label{eq;Y;app}
    \partial_tY_\epsilon-\Lambda^\gamma Y_\epsilon
    =|D|^\sigma\xi_\epsilon. \quad Y_\epsilon(0)=\int_{-\infty}^0 e^{-s\Delta}|D|^\sigma\xi_\epsilon\mathrm ds
\end{equation}
Define that
\[
X_\epsilon^{\RS{lr}}:=B_f(Y_\epsilon,Y_\epsilon)
=
\int_0^tP(t-s)f(D)(Y_\epsilon^2)\,\mathrm ds,
\]
Since $Y_\epsilon^2\in C^\infty(\m T)$, $B_f(Y_\epsilon,Y_\epsilon)$ has no zero Fourier mode. 
we rewrite the above expression as
\begin{equation}\label{eq;X Y;order-2}
X_\epsilon^{\RS{lr}}
=
f(D)Y_\epsilon^{\RS{lr}}(t)
-
f(D)P(t)Y_\epsilon^{\RS{lr}}(0),
\end{equation}
where
\[
Y_\epsilon^{\RS{lr}}(t)
=
\sum_{k\in\mathbb Z_*}
\int_{-\infty}^t
e^{-(t-s)|k|^\gamma}
\widehat{Y_\epsilon^2}(s,k)\,
\mathrm ds\,e^{ikx}.
\]
Here $\mathbb Z_*=\mathbb Z\setminus\{0\}$. Therefore, it suffices to show the convergence of $Y_\epsilon^{\RS{lr}}(t)$. The main objective of this section is to establish the convergence of 
\[
Y_\epsilon^{\RS{lr}}\in L^2(\Omega; C_TW^\alpha\cap C_T^\beta W^0)
\]
for $\gamma>\frac{3}{2}$.
 %Our work is based on Proposition \ref{prop;key computation} and following Wick's theorem
%\begin{proposition}\label{prop;Wick th}
	%Let $\ml T$ be a finite index set, and let $\{X_\alpha\}_{\alpha\in\ml T}$ be a a collection of real or complex-valued centered jointly Gaussian random variables. Then,
	%\begin{equation}
		%\m E \prod_{\alpha\in\ml T}X_\alpha=\sum_{P\in\ml P(\ml T)}\prod_{\{\alpha,\beta\}\in %P}\m E[X_\alpha X_\beta].
	%\end{equation}
%\end{proposition}
\begin{proposition}\label{prop;regularity;2 order}
	Let $\gamma\in(1,2]$, $\sigma\in[\frac{\gamma-1}{2},\frac{3\gamma-1}{4})$, then for any $T>0$, $\alpha<2\gamma-2\sigma-1$ and $\beta\in(0,1/2)$, there exist $Y^{\RS{lr}}\in C_TW^\alpha\cap C_T^{\beta}W^{\alpha-\beta\gamma}$ a.s. , such that $Y^{\RS{lr}}_\epsilon \to Y^{\RS{lr}}$ in $L^2(\Omega;C_TW^\alpha\cap C_T^{\beta}W^{\alpha-\beta\gamma})$.
	\begin{proof}
   \par Denote $\m Z_*=\m Z\backslash\{0\}$, $L=(l,m)\in \m Z_*^2$, $\ml Z=\m Z_*^2\backslash\{l+m=0\}$ $C_\epsilon(L)=\phi(\epsilon l)\phi(\epsilon m)$, $\psi(L)=l+m$, then $Y^{\RS{lr}}_\epsilon(t)$ can be written as 
   	\begin{equation}\label{eq;2-order;rep}
   		Y^{\RS{lr}}_{\epsilon}(t)=\sum_{L\in \ml Z}C_\epsilon(L)h_L(t)e_L(x).
   	\end{equation}
 where $e_L(x)=e^{i\psi(L)x}$, $h_L(t)$ satisfies
   	\begin{align*}
   		h_L(t)=\int_{-\infty}^t e^{-(t-\tau)\vert \psi(L)\vert^\gamma}Y_{l}Y_{m}\mathrm{d}\tau,
   	\end{align*}
   	Define  $H_{L,L'}(t,s)=\m E[h_L(t)h_{L'}(s)]$ and $\hat{H}_{L,L'}(t,s)=\m E[(h_{L}(t)-h_{L}(s))(h_{L'}(t)-h_{L'}(s))]$, $\psi(L)=l+m$.  According to Proposition \ref{prop;key computation}, it remains to establish two estimates for $H_{L,L'}(t,s)$ and $\hat{H}_{L,L'}(t,s)$ for $0<s<t\leq T$.
    Notice that 
    \begin{align*}
        H_{L,L'}(t,s)=\m E[\ml I_{L,L'}(Y_lY_m\otimes Y_{l'}Y_{m'})(t,s)].
    \end{align*}
    Applying Wick's theorem, we obtain
    \begin{equation}\label{eq;covariance;2-order}
        \m E[(Y_lY_m\otimes Y_{l'}Y_{m'})(\tau,r)]=|l|^{2\sigma-\gamma}|m|^{2\sigma-\gamma}2\pi(\delta_{l,-m'}\delta_{m,-l'}+\delta_{l,-l'}\delta_{m,-m'})e^{-\vert \tau-r\vert\vert l\vert^\gamma}e^{-\vert \tau-r\vert\vert m\vert^\gamma},
    \end{equation}
    where $\delta_{l,m}$ is the indicator function supported on ${l=m}$. Then we have $H_{L,L'}(t,t')\neq0$ only for $\psi(L)=-\psi(L')\neq0$. Moreover, Lemma \ref{le; H-hat;1} yields the estimate
    \begin{align*}
        |H_{L,L'}(t,s)|
        \lesssim&
        \frac{|l|^{2\sigma-\gamma}|m|^{2\sigma-\gamma}}{|l+m|^\gamma (|l+m|^\gamma +|l|^\gamma+|m|^\gamma)}
    \end{align*}
   We firstly consider $\sigma>\frac{\gamma-1}{2}$, $a=2\gamma-2\sigma-1-\varepsilon$, where $\varepsilon>0$ is sufficiently small. Then $0<a<\gamma$, $\gamma-2a<1$, $\gamma-2\sigma<1$, and
\begin{align*}
    a+(\gamma-2a)+(\gamma-2\sigma)=1+\varepsilon\in(1,2).
\end{align*}
Since
\begin{align*}
    |l+m|^\gamma+|l|^\gamma+|m|^\gamma\gtrsim (|l+m|+|m|)^a|l|^{\gamma-a},
\end{align*}
Corollary \ref{coro;summable} yields
\begin{align*}
    \sum_{m\in\m Z_*\backslash\{-l\}}\frac{1}{|l+m|^{\gamma-2a}(|l+m|+|m|)^a|m|^{\gamma-2\sigma}}\lesssim\frac{1}{|l|^\varepsilon}.
\end{align*}
Consequently,
\begin{align*}
    \sum_{\substack{l,m\in\m Z_*\\l+m\neq0}}\frac{|l+m|^{2a}|l|^{2\sigma-\gamma}|m|^{2\sigma-\gamma}}{|l+m|^\gamma(|l+m|^\gamma+|l|^\gamma+|m|^\gamma)}
    \lesssim\sum_{l\in\m Z_*}\frac{1}{|l|^{2\gamma-a-2\sigma+\varepsilon}}
    =\sum_{l\in\m Z_*}\frac{1}{|l|^{1+2\varepsilon}}<\infty.
\end{align*}
 
    For $\sigma=\frac{\gamma-1}{2}$, we have $\gamma-2\sigma=1$, $a=\gamma-\varepsilon$, then we have
\begin{align*}
    \sum_{\substack{l,m\in\m Z_*\\l+m\neq0}}\frac{|l+m|^{2a}}{|l+m|^\gamma(|l+m|^\gamma+|l|^\gamma+|m|^\gamma)|l||m|}
    \lesssim\sum_{\substack{l,m\in\m Z_*\\l+m\neq0}}\frac{|l+m|^{2a}}{|l+m|^{2\gamma-2\epsilon}|l|^{1+\varepsilon}|m|^{1+\varepsilon}}<\infty
\end{align*}
   Combining above estimates, we have for $\sigma\in[\frac{\gamma-1}{2},\frac{3\gamma-1}{4})$ and $a<2\gamma-2\sigma-1$,
  \begin{equation}\label{eq;sum;2-order}
             \begin{aligned}
\sum_{\substack{L,L'\in\ml Z\\\psi(L)=-\psi(L')\neq0}}\sup_{t\neq s\in[0,T]}\vert H_{L,L'}(t,s)\vert\psi(L)\vert^a\vert \psi(L')\vert^a<\infty,
   \end{aligned}
    \end{equation}
 \par We next turn to the covariance increment, which determines the temporal regularity of the stochastic process. A direct computation gives
   \begin{equation}\label{eq;relation;H}
   	\hat H_{L,L'}(t,s)=H_{L,L'}(t,t)+H_{L,L'}(s,s)-H_{L,L'}(t,s)-H_{L,L'}(s,t).
   \end{equation}
Noting that $H_{L,L'}(t,s)=\ml I_{L,L'}(\m E[Y_{l}Y_m\otimes Y_{l'}Y_{m'}])(t,s)$, by \eqref{eq;covariance;2-order} and applying Lemma \ref{le; H-hat;1} again, we obtain
  \begin{align*}
    |\widehat H_{L,L'}(t,s)|\lesssim |l|^{2\sigma-\gamma}|m|^{2\sigma-\gamma}\frac{1\wedge |l+m|^\gamma|t-s|}{|l+m|^\gamma(|l+m|^\gamma+|l|^\gamma+|m|^\gamma)}.
\end{align*}
Fix $\alpha<a<2\gamma-2\sigma-1$ and $0<\beta<1/2$. Choose $\beta<b\leq1/2$ such that $(b-\beta)\gamma<a-\alpha$. Since $\alpha-\beta\gamma+b\gamma<a$, \eqref{eq;sum;2-order} and the inequality $1\wedge x\leq x^{2b}$ yield
\begin{align*}
    \sum_{\substack{L,L'\in\ml Z\\\psi(L)=-\psi(L')}}|\psi(L)|^{\alpha-\beta\gamma}|\psi(L')|^{\alpha-\beta\gamma}\sup_{\substack{t,s\in[0,T]\\t\neq s}}\frac{|\widehat H_{L,L'}(t,s)|}{|t-s|^{2b}}<\infty.
\end{align*}
Applying Proposition \ref{prop;key computation} with $c=\alpha-\beta\gamma$, we obtain
\[
Y^{\RS{lr}}_\epsilon\to Y^{\RS{lr}}\quad\text{in }L^2(\Omega;C_TW^\alpha\cap C_T^\beta W^{\alpha-\beta\gamma}).
\]
	\end{proof}
\end{proposition}
\begin{remark}
 The restriction \(\sigma\geq(\gamma-1)/2\) isolates the regime in which the first-order stochastic convolution is not of positive spatial regularity. Indeed, if \(\sigma<(\gamma-1)/2\), then
\[
Y\in C_TW^r
\qquad\text{for every }0<r<\frac{\gamma}{2}-\frac12-\sigma.
\]
Since \(W^r\) is an algebra for \(r>0\), the product \(Y^2\) is well defined. Consequently, standard semigroup estimates yield
\[
Y^{\RS{lr}}\in C_TW^{\bar\alpha}
\qquad\text{for every }\bar\alpha<
\frac{3\gamma}{2}-\frac32-\sigma.
\]
Thus, the construction of the second-order term in this regime follows from ordinary multiplication and semigroup estimates. We therefore focus below on the singular regime \(\sigma\geq(\gamma-1)/2\).
\end{remark}
\begin{corollary}\label{coro;2-order}
   Let $\gamma\in(3/2,2]$, $\sigma=1/2$, then for any $\alpha<2\gamma-2$ and $\beta\in(0,1/2)$, there exists $Y^{\RS{lr}}\in C_TW^\alpha\cap C_T^\beta W^{\alpha-\frac{\gamma}{2}}$ a.s. such that $Y_\epsilon^{\RS{lr}} \to Y^{\RS{lr}}$ as $\epsilon \to 0$ in $L^2(\Omega;C_TW^\alpha\cap C_T^\beta W^{\alpha-\frac{\gamma}{2}})$.
\end{corollary}
\subsection*{Construction of $Y^{\RS{rLlr}}$}
\par Similar as the convergence of $X^{\RS{rLlr}}$, our target is focusing on the convergence of
\begin{equation}
	Y^{\RS{rLlr}}_\epsilon(t)=\int_{-\infty}^t P(t-s)\Pi_0^\perp(Y_\epsilon \cdot f(D)Y_\epsilon^{\RS{lr}})\mathrm{d}s.
\end{equation}
Let $Z_\epsilon^{\RS{lr}}=X_\epsilon^{\RS{lr}}-f(D)Y_\epsilon^{\RS{lr}}$ and $Z_\epsilon^{\RS{rLlr}}$ define as 
\begin{equation}\label{eq;Z;3-order}
    Z_\epsilon^{\RS{rLlr}}:=\int_{0}^tP(t-s)\Pi_0^\perp Y_\epsilon(s) Z_\epsilon^{\RS{lr}}(s)\mathrm ds
\end{equation}
Let 
\begin{equation}
    X_\epsilon^{\RS{rLlr}}=\int_0^t P(t-s)f(D)(Y_\epsilon(s) X^{\RS{lr}}_\epsilon(s)) \mathrm ds
\end{equation}
then $X_\epsilon^{\RS{rLlr}}$ can be written as 
\begin{equation}\label{eq;X Y;3-order}
\begin{aligned}
    X^{\RS{rLlr}}_\epsilon(t)=f(D)Y^{\RS{rLlr}}_\epsilon(t)-f(D)P(t)Y^{\RS{rLlr}}_\epsilon(0)
    +f(D)Z_\epsilon^{\RS{rLlr}}.
\end{aligned}
\end{equation}
Since $Z_\epsilon^{\RS{lr}}$ converge for any $t>0$ by the smooth effect of $P(t)$, we focus on the convergence of $Y_\epsilon^{\RS{rLlr}}$. The main strategy is similar as the above argument which use the convergence theorem \ref{prop;key computation} but more complex. For our convenience, we introduce some notation.
\par Let $L=(k,l,m)\in \m Z_*^3$, $C_\epsilon(L)=\phi(\epsilon k)\phi(\epsilon l)\phi(\epsilon m)$. Define the integral operator $\ml I_{L,L'}$ as 
\begin{equation}\label{eq;operator;I}
	\ml I_{L,L'}f(t,t')=\int_{-\infty}^{t}\int_{-\infty}^{t'} e^{-\vert \psi(L)\vert^\gamma(t-s)}e^{-\vert \psi(L')\vert^\gamma(t'-s')} f(s,s')\mathrm{d}s'\mathrm{d}s,
\end{equation} 
and define the tensor product as $f\otimes g(t,t')=f(t)g(t')$.  we rewrite $Y_\epsilon^{\RS{rLlr}}$ as 
\begin{equation}\label{eq;formula,3-order}
	Y_\epsilon^{\RS{rLlr}}=\sum_{L\in \mathbb{Z}_*^3}C_\epsilon(L) h_{L}(t)e_L(x),
\end{equation}
where $e_L(x)=e^{i(k+l+m)x}$, and $h_L(x)$ is defined as follows
\begin{equation}\label{eq;hL;3-order}
   h_L(t)= \int_{-\infty}^te^{-|\psi(L)|^\gamma(t-s)}Y_k(s)\int_{-\infty}^sf(l+m)e^{-|\psi(L_\downarrow)|^\gamma(s-\tau)}Y_l(\tau)Y_m(\tau)\mathrm d\tau\mathrm ds
\end{equation}
and vanish on the case of $k+l+m=0$ or $l+m=0$. Here $L_\downarrow=(l,m)$ and $L'_{\downarrow}=(l',m')$ is the tail of $L,L'$. 
\par We may replace $h_L$ by its symmetrization. More precisely, for $L=(k,l,m)$, let $\mk S L$ be the set of all permutations of its components:
\begin{equation}
    \mk S L=\{(k,l,m),(k,m,l),(l,k,m),(l,m,k),(m,k,l),(m,l,k)\}.
\end{equation}
Define
\begin{equation}
    \widetilde h_L(t)=\frac{1}{|\mk S L|}\sum_{N\in\mk S L}h_N(t).
\end{equation}
Since $C_\epsilon(N)=C_\epsilon(L)$ and $e_N=e_L$ for every $N\in\mk S L$, replacing $h_L$ by $\widetilde h_L$ leaves the random Fourier series unchanged.
\eqref{eq;formula,3-order} yields that
\begin{equation}\label{eq;formula,3-order;redefined}
    Y_\epsilon^{\RS{rLlr}}(t)=\sum_{L\in\m Z_*^3}C_\epsilon(L)\widetilde h_L(t)e_L(x).
\end{equation}
In the following, we work with these symmetrized coefficients and set
\begin{align*}
    H_{L,L'}(t,t')=\m E\widetilde h_L(t)\widetilde h_{L'}(t')],\qquad \hat H_{L,L'}(t,t')=\m E[(\widetilde h_L(t)-\widetilde h_L(t'))(\widetilde h_{L'}(t)-\widetilde h_{L'}(t'))].
\end{align*}
It then follows from \eqref{eq;operator;I}, \eqref{eq;hL;3-order} and Proposition \ref{prop;resonant convergence} that
\begin{equation}\label{eq;H}
    H_{L,L'}(t,t')=\frac{1}{|SL||SL'|}\ml I_{L,L'}\sum_{N\in SL,N'\in SL'}R_{N,N'}(t,t')
\end{equation}
Moreover, \eqref{eq;R;decomposition} yields that
 \begin{equation}\label{eq;H;decomposition}
        \begin{aligned}
              H_{L,L'}(t,t')=\sum_{i=1}^3H^i_{L,L'}(t,t')
        \end{aligned}
        \end{equation}
        where
\begin{equation}\label{eq;H;pair;decomposition}
    H^i_{L,L'}(t,t')=\frac{1}{|SL||SL'|}\ml I_{L,L'}R^i_{L,L'}(t,t').
\end{equation}
Define
\begin{equation}\label{eq;H;pair}
    H_{L,L'}(P_i;t,t')=\ml I_{L,L'}R_{L,L'}(P_i;\cdot,\cdot)(t,t').
\end{equation}
Similarly, define
\begin{equation}\label{eq;hat H}
     \hat H^i_{L,L'}(t,t')= H^i_{L,L'}(t,t)+ H^i_{L,L'}(t',t')- H^i_{L,L'}(t,t')- H^i_{L,L'}(t',t)
\end{equation}

\begin{proposition}\label{prop;regularity;3 order}
	Let $\gamma\in(1,2]$. Let $\beta\in(0,1/2)$, $\gamma$, $\sigma$, $\alpha$  satisfying following conditions
    \begin{itemize}
        \item {If $\sigma\in[\gamma-1,\min(\gamma-\frac{3}{4},\frac{3(\gamma-1)}{2}))$, then  $\alpha<(7\gamma-6\sigma-5)/2$.}
        \item{If $\sigma\in[\frac{\gamma-1}{2},\gamma-1)$, then $\alpha<(3\gamma-2\sigma-1)/2$.}
    \end{itemize}
    then there exists $Y^{\RS{rLlr}}\in C_TW^\alpha\cap C_T^\beta W^{\alpha-\frac{\gamma}{2}}$ a.s. such that $Y_\epsilon^{\RS{rLlr}} \to Y^{\RS{rLlr}}$ as $\epsilon \to 0$ in $L^2(\Omega;C_TW^\alpha\cap C_T^\beta W^{\alpha-\frac{\gamma}{2}})$.
	\begin{proof}
		According to Proposition \ref{prop;key computation}, it remains to establish the covariance and covariance increment estimates for $H_{L,L'}(t,t')$ and $\hat H_{L,L'}(t,t')$ for $t,t'\in[0,T]$ with $t\neq t'$. By \eqref{eq;H;decomposition}, we estimate $H^i_{L,L'}$ and its covariance increment separately for $i=1,2,3$. At the endpoint $\sigma=(\gamma-1)/2$, the nonnegative summands in these estimates are bounded by their counterparts with $\sigma$ replaced by $\sigma+\eta$, where $\eta>0$ is sufficiently small. Since the restriction on the spatial regularity exponent is strict, the endpoint case follows from the non-endpoint estimates. We therefore assume below that $\sigma>(\gamma-1)/2$.

\begin{itemize}
     \item{\textbf{For $H^1_{L,L'}(t,t')$}
     \par We have the fact that $\vert \psi (L)\vert=\vert \psi (L')\vert=\vert k+l+m\vert$. Since for $N\in \mk S L$, we have $\psi(L)=\psi(N)$. It then follows from Proposition \ref{prop;resonant convergence},\eqref{eq;H;pair;decomposition} and \eqref{eq;H;pair}
     \begin{equation}\label{eq;H;pair;1}
    H^1_{L,L'}(t,t')=\frac{2}{|\mk S L||\mk S L'|}I_{L,L'}\sum_{\substack{N\in\mk S L\\N'\in\mk S L'}}R_{N,N'}(P_1;t,t')=\frac{2}{|\mk S L||\mk S L'|}\sum_{\substack{N\in\mk S L\\N'\in\mk S L'}}H_{N,N'}(P_1;t,t')
\end{equation}
     By Assumption \ref{ass:g}, Lemma \ref{le; H-hat;1} and Proposition \ref{prop;resonant convergence}, we have 
\begin{align*}
	\vert H_{L,L'}(P_1;T,T')\vert
	\leq&\frac{1}{\vert k+l+m\vert^{\gamma}(\vert k+l+m\vert^{\gamma}+|l+m|^\gamma+\vert k\vert^\gamma)}\cdot \frac{\vert (l+m)\vert^2\vert k\vert^{2\sigma-\gamma}\vert l\vert^{2\sigma-\gamma}\vert m\vert^{2\sigma-\gamma}}{\vert l+m\vert^\gamma(\vert l+m\vert^\gamma+(\vert l\vert^\gamma+\vert m\vert^\gamma ))},
\end{align*}
and 
\begin{align*}
	\vert \hat H_{L,L'}(P_1;T,T')\vert\leq\frac{1\wedge\vert k+l+m\vert^\gamma \vert T-T'\vert}{\vert k+l+m\vert^{\gamma}(\vert k+l+m\vert^{\gamma}+|l+m|^\gamma+\vert k\vert^\gamma)}\frac{\vert (l+m)\vert^2\vert k\vert^{2\sigma-\gamma}\vert l\vert^{2\sigma-\gamma}\vert m\vert^{2\sigma-\gamma}}{\vert l+m\vert^\gamma(\vert l+m\vert^\gamma+(\vert l\vert^\gamma+\vert m\vert^\gamma ))},
\end{align*}
for any $\overline{T}>0$. We first consider $\frac{\gamma-1}{2}<\sigma<\gamma-1$.
Since $1<2\gamma-2\sigma-1$, \eqref{eq;sum;2-order} yields that 
\[
\sum_{\substack{l,m\in\mathbb Z_*\\l+m\neq0}}
\frac{|l+m|^2|l|^{2\sigma-\gamma}|m|^{2\sigma-\gamma}}
{|l+m|^\gamma
 (|l+m|^\gamma+|l|^\gamma+|m|^\gamma)}
<\infty,
\]
It therefore suffices to show that
\begin{equation}\label{eq;sum;k}
\sum_{k\notin\{0,-l-m\}}
\frac{|k+l+m|^{2a}|k|^{2\sigma}}
{|k+l+m|^{2\gamma}|k|^\gamma}
\end{equation}
is bounded uniformly for $l+m\neq0$. Let $a$ sufficiently close to
$(3\gamma-2\sigma-1)/2$ from below. Since $\sigma\in(\frac{\gamma-1}{2},\gamma-1)$, we have for small enough $\varepsilon$, we have
\begin{align*}
    0<\gamma-2\sigma<1,\qquad  \varepsilon>2\gamma-2a>3\gamma-2\sigma-2a-1>0
\end{align*}
then Lemma \ref{le;summary criterion} yields that
\[
\sum_{k\notin\{0,-l-m\}}
\frac{|k+l+m|^{2a}|k|^{2\sigma}}
{(|k+l+m|^\gamma+|k|^\gamma)|k+l+m|^\gamma|k|^\gamma}
\lesssim \frac{1}{|l+m|^{3\gamma-2\sigma-2a-1}}\leq1.
\]
This proves the desired summability.
\par Next, suppose that $\sigma\geq\gamma-1$, and let
\[
\delta=\sigma-\gamma+1,\qquad
0\leq\delta<
\min\left\{\frac14,\frac{\gamma-1}{2}\right\}.
\]
It suffices to consider
\[
a=\frac{7\gamma-6\sigma-5}{2}-\varepsilon,
\qquad 0<\varepsilon<1-3\delta,
\]
with $\varepsilon$ arbitrarily small. Then
\[
3\gamma-2\sigma-2a=1+4\delta+2\varepsilon.
\]
Choose $0<\eta<2\varepsilon$. Since
$|l+m|\leq|k+l+m|+|k|$, we have
\[
|k+l+m|^\gamma+|l+m|^\gamma+|k|^\gamma
\gtrsim
|l+m|^{4\delta+\eta}
(|k+l+m|+|k|)^{\gamma-4\delta-\eta}.
\]
Consequently,
\[
\begin{aligned}
&\sum_{k\notin\{0,-l-m\}}
\frac{|k+l+m|^{2a}|k|^{2\sigma}}
{|k+l+m|^{\gamma}|k|^{\gamma}(|k+l+m|^\gamma+|l+m|^\gamma+|k|^\gamma)}\\
&\qquad\lesssim
|l+m|^{-4\delta-\eta}
\sum_{k\notin\{0,-l-m\}}
\frac{|k+l+m|^{2a}|k|^{2\sigma}}
{(|k+l+m|+|k|)^{\gamma-4\delta-\eta}|k+l+m|^{\gamma}|k|^{\gamma}}.
\end{aligned}
\]
Here since $\gamma>1$ and $\sigma=\gamma-1+\delta$, we have
\[
\gamma+(\gamma-4\delta-\eta)-2a=\gamma-1+2\delta+2\epsilon-\eta>2\epsilon-\eta,
\]
and since $\epsilon<1-3\delta$, we have
\[
\gamma+(\gamma-4\delta-\eta)-2\sigma=2-6\delta-\eta>2\epsilon-\eta.
\]
Noting that 
\begin{align*}
 2\gamma+(\gamma-4\delta-\eta)-2a-2\sigma=1+2\epsilon-\eta,
\end{align*}
then splitting the $k$-sum as above and Corollary \ref{coro;summable} yields
\[
\sum_{k\notin\{0,-l-m\}}
\frac{|k+l+m|^{2a}|k|^{2\sigma}}
{(|k+l+m|+|k|)^{\gamma-4\delta-\eta}|k+l+m|^{\gamma}|k|^{\gamma}}
\lesssim \frac{1}{|l+m|^{2\varepsilon-\eta}}\leq1.
\]
Thus \eqref{eq;sum;k} is bounded by
$C|l+m|^{-4\delta-\eta}$.

Finally, since $\delta<\min(\frac{1}{4},\frac{\gamma-1}{2})$ then
\[
(2-4\delta-\eta)/2
<2\gamma-2\sigma-1.
\]
Hence the proof of Proposition \ref{prop;order-2} shows that
\[
\sum_{\substack{l,m\in\mathbb Z_*\\l+m\neq0}}
\frac{
|l+m|^{2-4\delta-\eta}
|l|^{2\sigma-\gamma}|m|^{2\sigma-\gamma}
}{
|l+m|^\gamma
(|l+m|^\gamma+|l|^\gamma+|m|^\gamma)
}
<\infty.
\]
Combining the above estimates, by \eqref{eq;H;pair;1}, we obtain the same weighted summability bound for $H^1_{L,L'}$. Indeed, each term in its finite expansion reduces to the representative estimate for $H_{L,L'}(P_1;t,t')$ by reindexing, since for any $\bar L\in \mk S L$, we have $\psi(L)=\psi(\bar L)$ and the same identities hold for $L'$.
\begin{align*}
	\sum_{\substack{L,L'\in \mathbb{Z}^3_*\\\psi(L)=-\psi(L')\neq 0}}&\vert \psi(L)\vert^{a}|\psi(L')|^a \sup_{0<T\leq\overline{T}}\vert H^1_{L,L'}(T,T)\vert<\infty.
\end{align*}

\par We next turn to the covariance increment. Using the same argument when estimates $H_{L,L'}(P_1;T,T)$, for $\beta\in(0,1/2)$ , $c=a-\frac{\gamma}{2}$, we have

\begin{align*}
	\sum_{\substack{L,L'\in \mathbb{Z}^3_*\\\psi(L_\downarrow)\neq0,\psi(L'_\downarrow)\neq0\\\psi(L)=-\psi(L')\neq 0}}&\vert \psi(L)\vert^{a-\frac{\gamma}{2}}|\psi(L')|^{a-\frac{\gamma}{2}}\sup_{0<t,t'\leq T}\frac{\vert \widehat H^1_{L,L'}(t,t')\vert}{\vert t-t'\vert}<\infty.
\end{align*}}
     \item{\textbf{For $ H^2_{L,L'}(T,T')$}
     \par The fact that $\vert \psi(L)\vert=\vert \psi(L')\vert=\vert k+l+m\vert$ is also valid. It then follows from \eqref{eq;H} and \eqref{eq;H;pair}
      \begin{equation}\label{eq;H;pair;2}
        H^2_{L,L'}(t,t')=\frac{4}{|\mk S L||\mk S L'|}I_{L,L'}\sum_{\substack{N\in\mk S L\\N'\in\mk S L'}}R_{N,N'}(P_2;t,t')=\frac{4}{|\mk S L||\mk S L'|}\sum_{\substack{N\in\mk S L\\N'\in\mk S L'}}H_{N,N'}(P_2;t,t').
\end{equation}
    By Lemma \ref{le; H-hat;1} and proof in Proposition \ref{prop;resonant convergence}, we have
\begin{align*}
     \vert H_{L,L'}(P_2;t,t')\vert\leq&\frac{|k|^{2\sigma}}{|k+l+m|^\gamma(|k+l+m|^\gamma+|k|^\gamma\wedge|l|^\gamma)|k|^{\gamma}|l|^\gamma|m|^\gamma}\\
     &\cdot\frac{|l+m||k+m||l|^{2\sigma}|m|^{2\sigma}}{(\vert l+m\vert^\gamma+\vert l\vert^\gamma)(\vert k+m\vert^\gamma+\vert k\vert^\gamma+|m|^\gamma)\wedge(\vert l+m\vert^\gamma+\vert l\vert^\gamma+|m|^\gamma)(\vert k+m\vert^\gamma+\vert k\vert^\gamma)}
\end{align*}
and 
\begin{align*}
    |\hat H_{L,L'}(P_2;t,t')|\lesssim& \frac{|k|^{2\sigma}(1\wedge |k+l+m|^\gamma|t-t'|)}{|k+l+m|^\gamma(|k+l+m|+|k|^\gamma\wedge|l|^\gamma)|k|^{\gamma}|l|^\gamma|m|^\gamma}\\
   &\cdot\frac{|l+m||k+m||l|^{2\sigma}|m|^{2\sigma}}{(\vert l+m\vert^\gamma+\vert l\vert^\gamma)(\vert k+m\vert^\gamma+\vert k\vert^\gamma+|m|^\gamma)\wedge(\vert l+m\vert^\gamma+\vert l\vert^\gamma+|m|^\gamma)(\vert k+m\vert^\gamma+\vert k\vert^\gamma)}.
\end{align*}
The two expressions appearing in the minimum in the
denominator are comparable. Indeed,
\[
|m|^\gamma\lesssim |l+m|^\gamma+|l|^\gamma,
\qquad
|m|^\gamma\lesssim |k+m|^\gamma+|k|^\gamma.
\]
Thus either expression may be used in the following
upper bounds, up to a constant.
\par We first consider $\sigma>\gamma-1$.
By symmetry in $k$ and $l$, it suffices to consider
$|l|\geq|k|$. For $d\in[0,1]$, we have
\[
|k+l+m|^\gamma+|k|^\gamma
\geq |k+l+m|^{d\gamma}|k|^{(1-d)\gamma}.
\]
Moreover,
\[
\frac{|l+m|}{|l+m|^\gamma+|l|^\gamma}
\lesssim |l|^{1-\gamma},
\qquad
\frac{|k+m|}
{|k+m|^\gamma+|k|^\gamma+|m|^\gamma}
\lesssim |m|^{1-\gamma}.
\]
Consequently,
\begin{equation}\label{eq;H2;1}
    \begin{aligned}
&|k+l+m|^{2a}\sup_{0\leq t,t'\leq T}|H_{L,L'}(P_2;t,t')|\\
&\qquad\lesssim
\frac{1}{
|k+l+m|^{(1+d)\gamma-2a}
|l|^{2\gamma-2\sigma-1}
|k|^{(2-d)\gamma-2\sigma}
|m|^{2\gamma-2\sigma-1}}.
\end{aligned}
\end{equation}
We may now enlarge the summation domain by dropping
the restriction $|l|\geq|k|$.

Write
\[
\sigma=\gamma-1+\delta,\qquad 0<\delta<\frac14.
\]
It suffices to consider $a$ arbitrarily close to the
claimed upper bound. Set
\[
a=\frac{7\gamma-6\sigma-5}{2}-\varepsilon,
\qquad 0<\varepsilon<\frac{1-4\delta}{2},
\]
and choose $d\in(0,1)$ such that
\[
1-2\delta<d\gamma<
\min\{\gamma,\,2-6\delta-2\varepsilon\}.
\]
This interval is nonempty because $\gamma>1$ and $\delta<1/4$. We first sum over $l$. The two exponents are
\[
(1+d)\gamma-2a=d\gamma-1+6\delta+2\varepsilon,
\qquad
2\gamma-2\sigma-1=1-2\delta.
\]
Both belong to $(0,1)$, and their sum is greater than $1$.
Lemma \ref{le;convolution decay} therefore gives
\[
\sum_{l\notin\{0,-k-m\}}
\frac{1}{
|k+l+m|^{(1+d)\gamma-2a}
|l|^{2\gamma-2\sigma-1}}
\lesssim
\frac{1}{|k+m|^{(3+d)\gamma-2a-2\sigma-2}}.
\]
Here $k+m\neq0$, as required by the pairing.

Next, since
\[
(3+d)\gamma-2a-2\sigma-2
=d\gamma-1+4\delta+2\varepsilon,
\]
and
\[
(2-d)\gamma-2\sigma=2-d\gamma-2\delta.
\]
They also belong to $(0,1)$, and their sum equals
$1+2\delta+2\varepsilon$. Applying
Lemma \ref{le;convolution decay} again yields
\[
\sum_{k\notin\{0,-m\}}
\frac{1}{
|k+m|^{(3+d)\gamma-2a-2\sigma-2}
|k|^{(2-d)\gamma-2\sigma}}
\lesssim
\frac{1}{|m|^{5\gamma-2a-4\sigma-3}}.
\]
Combining this with the remaining factor gives
\[
\sum_{m\neq0}
\frac{1}{|m|^{7\gamma-2a-6\sigma-4}}
=
\sum_{m\neq0}\frac{1}{|m|^{1+2\varepsilon}}
<\infty.
\]

For the endpoint $\sigma=\gamma-1$, set
\[
a=\frac{\gamma+1}{2}-\varepsilon,
\qquad 0<\varepsilon<\frac12,
\]
and choose
\[
1<d\gamma<\min\{\gamma,\,2-2\varepsilon\}.
\]
The endpoint convolution estimate gives
\[
\begin{aligned}
\sum_{l\notin\{0,-k-m\}}
\frac{1}{|k+l+m|^{d\gamma-1+2\varepsilon}|l|}
&\lesssim
\frac{\log(2+|k+m|)}
{|k+m|^{d\gamma-1+2\varepsilon}}\\
&\lesssim
\frac{1}{|k+m|^{d\gamma-1+\varepsilon}}.
\end{aligned}
\]
Since $d\gamma-1+\varepsilon$ and $2-d\gamma$
belong to $(0,1)$ and have sum $1+\varepsilon$,
the subsequent $k$-sum is bounded by
$C|m|^{-\varepsilon}$. The final sum is therefore
bounded by
\[
\sum_{m\neq0}|m|^{-1-\varepsilon}<\infty.
\]

We now consider
\[
\frac{\gamma-1}{2}<\sigma<\gamma-1.
\]
Since $\gamma\leq2$, we have $\gamma>2\sigma$.
By symmetry, we work on the region $|k|\geq|l|$.
Using
\[
\frac{|l+m|}
{|l+m|^\gamma+|l|^\gamma}
\lesssim |m|^{1-\gamma},
\qquad
\frac{|k+m|}
{|k+m|^\gamma+|k|^\gamma+|m|^\gamma}
\leq |k+m|^{1-\gamma},
\]
we obtain
\begin{equation}\label{eq;H2;2}
    \begin{aligned}
&|k+l+m|^{2a}|H_{L,L'}(P_2;t,t')|\\
&\qquad\lesssim
\frac{
|k|^{2\sigma-\gamma}
|l|^{2\sigma-\gamma}
|m|^{2\sigma-\gamma}
|k+l+m|^{2a}
}{
|m|^{\gamma-1}|k+m|^{\gamma-1}
(|k+l+m|+|l|)^\gamma|k+l+m|^\gamma
}.
\end{aligned}
\end{equation}
We again enlarge the summation domain.

It suffices to consider
\[
a=\frac{3\gamma-2\sigma-1}{2}-\varepsilon
\]
with $\varepsilon>0$ sufficiently small. Then
\[
\gamma-2\sigma<1,\qquad
\gamma-2a<1,\qquad
1<3\gamma-2a-2\sigma=1+2\varepsilon<2.
\]
Corollary \ref{coro;summable} implies
\[
\sum_{l\notin\{0,-k-m\}}
\frac{|l|^{2\sigma-\gamma}|k+l+m|^{2a}}
{(|k+l+m|+|l|)^\gamma|k+l+m|^\gamma}
\lesssim
|k+m|^{-(3\gamma-2a-2\sigma-1)}
=
|k+m|^{-2\varepsilon}
\leq1.
\]
Furthermore,
\[
\gamma-2\sigma>0,\qquad
\gamma-1>0,\qquad
2\gamma-2\sigma-1>1.
\]
Hence Lemma \ref{le;summary criterion} gives
\[
\sup_{m\neq0}
\sum_{k\notin\{0,-m\}}
\frac{|k|^{2\sigma-\gamma}}{|k+m|^{\gamma-1}}
<\infty.
\]
Finally,
\[
\sum_{m\neq0}
\frac{|m|^{2\sigma-\gamma}}{|m|^{\gamma-1}}
=
\sum_{m\neq0}|m|^{-(2\gamma-2\sigma-1)}
<\infty.
\]
Combining these estimates, \eqref{eq;H2;1} and \eqref{eq;H2;2} yield that
\begin{equation}
    \begin{aligned}
	\sum_{\substack{L,L'\in \mathbb{Z}^3_*\\\psi(L_\downarrow)\neq0,\psi(L'_\downarrow)\neq0\\\psi(L)=-\psi(L')\neq 0}}&\vert \psi(L)\vert^{a}|\psi(L')|^a \sup_{0<t',t\leq\overline{T}}\vert H_{L,L'}(P_2;t,t')\vert<\infty.
\end{aligned}
\end{equation}
Combining the above estimates, by \eqref{eq;H;pair;2}, we obtain the same weighted summability bound for $H^2_{L,L'}$. Indeed, each term in its finite expansion reduces to the representative estimate for $H_{L,L'}(P_2;t,t')$ by reindexing, since for any $\bar L\in \mk S L$ $\psi(L)=\psi(\bar L)$ and the same identities hold for $L'$.
\begin{equation}
    \begin{aligned}
	\sum_{\substack{L,L'\in \mathbb{Z}^3_*\\\psi(L_\downarrow)\neq0,\psi(L'_\downarrow)\neq0\\\psi(L)=-\psi(L')\neq 0}}&\vert \psi(L)\vert^{a}|\psi(L')|^a \sup_{0<t',t\leq T}\vert H^2_{L,L'}(t,t')\vert<\infty.
\end{aligned}
\end{equation}
}
     \item{ \textbf{For $H^3_{L,L'}(t,t')$}
     \par Under the pairing $P_3$, to estimate a representative symmetrized contribution, let $L=(k,-k,m)$ and $L'=(k',-k',-m)$. Then $|\psi(L)|=|\psi(L')|=|m|$, independently of $k$ and $k'$. Then  we have
      \begin{equation}\label{eq;H;pair;3}
      \begin{aligned}
          H^3_{L,L'}(t,t')=&\frac{1}{|\mk SL||\mk SL'|}\sum_{\substack{N\in\mk S L\\N'\in\mk S L'}}I_{N,N'}\Bigl(R_{N,N'}(P_3;t,t')+R_{SN,N'}(P_3;t,t')\\&+R_{N,SN'}(P_3;t,t')+R_{SN,SN'}(P_3;t,t')\Bigr)\\
          =&\frac{1}{|\mk SL||\mk SL'|}\sum_{\substack{N\in\mk S L\\N'\in\mk S L'}} H^S_{N,N'}(P_3;t,t').
      \end{aligned}
\end{equation}
 It follows from \eqref{eq;R;3}, Lemma \ref{le;package} with $\theta=\gamma-1$ and  H\"older's inequality, that
\begin{align*}
					\Bigl |R_{L,L'}(P_3;t,t')+&R_{SL,L'}(P_3;t,t')+R_{L,SL'}(P_3;t,t')+R_{SL,SL'}(P_3;t,t')\Bigr|\\
					\lesssim&\vert kk'm\vert^{2\sigma-\gamma}\int_{-\infty}^{t}\int_{-\infty}^{t'}|\ml H_{k,m}(t,s)||\ml H_{k',-m}(t',s')|e^{-\vert m\vert^\gamma\vert s-s'\vert}\mathrm{d}s\mathrm{d}s'\\
					\lesssim&\vert kk'm\vert^{2\sigma-\gamma}|m|^{2(\gamma-1)}(|k+m|+|k|)^{2-\gamma}(|k'+m|+|k'|)^{2-\gamma}\\
					&\cdot\int_{-\infty}^{t}\int_{-\infty}^{t'}e^{-c(\vert k+m\vert+\vert k\vert)^\gamma(t-s)}e^{-c(\vert k'+m\vert+\vert k'\vert)^\gamma(t'-s')}e^{-\vert m\vert^\gamma\vert s-s'\vert}\mathrm{d}s\mathrm{d}s'\\
                    \lesssim& \frac{\vert kk'm\vert^{2\sigma-\gamma}|m|^{2(\gamma-1)}(|k+m|+|k|)^{2-\gamma}(|k'+m|+|k'|)^{2-\gamma}}{(\vert k+m\vert+\vert k\vert)^{\gamma}(\vert k'+m\vert+\vert k'\vert)^{\gamma}}.
				\end{align*}
	Then Lemma \ref{le; H-hat;1} and \eqref{eq;H} yields that 
    \begin{align*}
       |H^S_{L,L'}(P_3;t,t')|
         \lesssim&  \frac{\vert kk'm\vert^{2\sigma-\gamma}}{(\vert k+m\vert+\vert k\vert)^{2\gamma-2}(\vert k'+m\vert+\vert k'\vert)^{2\gamma-2}}\cdot \frac{1}{|m|^{2}}.
    \end{align*}
    Define $\widehat{H}^S_{N,N'}(P_3;t,t')=H^S_{N,N'}(P_3;t,t)+H^S_{N,N'}(P_3;t',t')-H^S_{N,N'}(P_3;t,t')-H^S_{N,N'}(P_3;t',t)$, we have
			 \begin{align*}
         |\widehat{H}^S_{N,N'}(P_3;t,t')|\lesssim  \frac{\vert kk'm\vert^{2\sigma-\gamma}}{(\vert k+m\vert+\vert k\vert)^{2\gamma-2}(\vert k'+m\vert+\vert k'\vert)^{2\gamma-2}}\cdot \frac{1\wedge|m|^\gamma|t-t'|}{|m|^{2}}
    \end{align*} 
 We first consider $\sigma>\gamma-1$. Write $\sigma=\gamma-1+\delta$, where $0\leq\delta<(\gamma-1)/2$. Then
\begin{align*}
    \gamma-2\sigma<1,\qquad 2>3\gamma-2-2\sigma=\gamma-2\delta>1.
\end{align*}
By the same argument as in the proof of Corollary \ref{coro;summable}, we obtain
\begin{align*}
    \sum_{k\in\m Z_*}\frac{|k|^{2\sigma}}{(|k+m|+|k|)^{2\gamma-2}|k|^\gamma}\lesssim\frac{1}{|m|^{3\gamma-2\sigma-3}}.
\end{align*}
Similarly,
\begin{align*}
    \sum_{k'\in\m Z_*}\frac{|k'|^{2\sigma}}{(|k'+m|+|k'|)^{2\gamma-2}|k'|^\gamma}\lesssim\frac{1}{|m|^{3\gamma-2\sigma-3}}.
\end{align*}
Consequently, for $a<(7\gamma-6\sigma-5)/2$, we have
\begin{align*}
    \sum_{\substack{L,L'\in\m Z_*^3\\L\overset{P_3}{\sim}L'}}|m|^{2a}\sup_{0<t,t'\leq T}|H^S_{L,L'}(P_3;t,t')|\lesssim\sum_{m\in\m Z_*}\frac{1}{|m|^{7\gamma-6\sigma-4-2a}}<\infty.
\end{align*}

For $\sigma\in((\gamma-1)/2,\gamma-1]$, we have
\begin{align*}
    \gamma-2\sigma<1,\qquad 3\gamma-2\sigma-2\geq\gamma>1.
\end{align*}
The same summation estimate therefore applies, and
\begin{align*}
    \sum_{\substack{L,L'\in\m Z_*^3\\L\overset{P_3}{\sim}L'}}|m|^{2a}\sup_{0<t,t'\leq T}|H^S_{L,L'}(P_3;t,t')|
    \lesssim&\sum_{m\in\m Z_*}\frac{1}{|m|^{2+\gamma-2\sigma-2a}}\sum_{k\in\m Z_*}\frac{|k|^{2\sigma}}{(|k+m|+|k|)^{2\gamma-2}|k|^\gamma}\\
    &\cdot\sum_{k'\in\m Z_*}\frac{|k'|^{2\sigma}}{(|k'+m|+|k'|)^{2\gamma-2}|k'|^\gamma}\\
    \lesssim&\sum_{m\in\m Z_*}\frac{1}{|m|^{7\gamma-6\sigma-4-2a}}<\infty.
\end{align*}
Since permutations preserve the output frequencies, the triangle inequality and reindexing yield
\begin{align*}
    &\sum_{\substack{L,L'\in\ml Z\\\psi(L)=-\psi(L')}}|\psi(L)|^a|\psi(L')|^a\sup_{0<t,t'\leq T}|H^3_{L,L'}(t,t')|\\
    &\leq\sum_{\substack{L,L'\in\ml Z\\\psi(L)=-\psi(L')}}\frac{|\psi(L)|^a|\psi(L')|^a}{|\mk S L||\mk S L'|}\sum_{\substack{N\in\mk S L\\N'\in\mk S L'}}\sup_{0<t,t'\leq T}|H^S_{N,N'}(P_3;t,t')|\\
    &\lesssim\sum_{\substack{N,N'\in\ml Z\\N\overset{P_3}{\sim}N'}}|\psi(N)|^a|\psi(N')|^a\sup_{0<t,t'\leq T}|H^S_{N,N'}(P_3;t,t')|<\infty.
\end{align*}
Here we used the uniformly bounded multiplicities of the permutations and the fact that $H^S_{N,N'}(P_3;t,t')$ vanishes unless $N\overset{P_3}{\sim}N'$. The same argument applies to the covariance increments.
Similarly, for $0<\beta\leq1/2$, the covariance increment estimate and the inequality $1\wedge x\leq x^{2\beta}$ yield
\begin{align*}
    \sum_{\substack{L,L'\in\m Z_*^3\\L\overset{P_3}{\sim}L'}}|m|^{2a-\gamma}\sup_{\substack{0<t,t'\leq T\\t\neq t'}}\frac{|\widehat H^S_{L,L'}(P_3;t,t')|}{|t-t'|}\lesssim\sum_{m\in\m Z_*}\frac{1}{|m|^{7\gamma-6\sigma-4-2a}}<\infty.
\end{align*}
Applying the same permutation argument to the covariance increments gives the corresponding weighted summability estimate for $\widehat H^3_{L,L'}$.}
\end{itemize}
Finally, combining these estimates with those for $P_1$ and $P_2$ verifies the required covariance bounds.
	\end{proof}	
\end{proposition}
\begin{remark}
    Similarly, the two regimes for $\sigma$ correspond to which term attains the minimum in \eqref{eq;r}. By Proposition \ref{prop;order-2}, the spatial regularity threshold of $f(D)Y^{\RS{lr}}$ is $2\gamma-2\sigma-2$. When $\sigma>\gamma-1$, both factors have negative regularity, so their sum determines the minimum. When $\frac{\gamma-1}{2}<\sigma<\gamma-1$, the second-order term has positive regularity, so the regularity of $Y$ determines the minimum. Thus, in both cases, the regularity of $Y^{\RS{rLlr}}$ is consistent with \eqref{eq;r}.
\end{remark}
\begin{corollary}\label{coro;3-order}
   Let $\gamma\in(3/2,2]$, $\sigma=1/2$, then for any $\alpha<(3\gamma-2)/2$ and $\beta\in(0,1/2)$, there exists $Y^{\RS{rLlr}}\in C_TW^\alpha\cap C_T^\beta W^{\alpha-\frac{\gamma}{2}}$ a.s. such that $Y_\epsilon^{\RS{rLlr}} \to Y^{\RS{rLlr}}$ as $\epsilon \to 0$ in $L^2(\Omega;C_TW^\alpha\cap C_T^\beta W^{\alpha-\frac{\gamma}{2}})$.
\end{corollary}
\par For higher order expansion like $Y_\epsilon^{\RS{rLrLlr}}$ and $Y_\epsilon^{\RS{{LL{lr}}{RR{lr}}}}$, it can be definition by para-product decomposition by Lemma \ref{le;est;paraproduct} and obviously converge in $L^2(\Omega;C_TH_p^{\frac{1}{2}+})$ We summarize the regularity of Gaussian trees as following
	Finally we re-centering the Gaussian trees by following corollary.

\begin{proof}[Proof of Proposition \ref{prop;order-2} and \ref{prop;order-3}]
			Let $\gamma>3/2$ and $\sigma=1/2$, then we have
            \begin{align*}
                2\gamma-2>1,\quad (3\gamma-2)/2>5/4, 
            \end{align*}
it then follows from Corollary \ref{coro;2-order} and \ref{coro;3-order} that, for any $T>0$ there exists $\delta>0$ such that we have  $Y\in C_TW^{-1/4+\delta}$, $Y^{\RS{lr}}\in C_TW^{1+\delta}$ and $Y^{\RS{rLlr}}\in C_TW^{5/4+\delta}$. 
\begin{itemize}
    \item {\textbf{For $X^{\RS{lr}}$}
    \par According to Assumption \ref{ass:g}, we have 
\begin{align*}
    f(D)Y_\epsilon^{\RS{lr}}\to f(D)Y^{\RS{lr}},\qquad as \; \epsilon\to0
\end{align*}
in  $L^2(\Omega; C_TW^{\delta})$ for any $t\in[0,T]$. Then \eqref{eq;X Y;order-2} yields that $X_\epsilon^{\RS{lr}}\to X^{\RS{lr}}$ in $L^2(\Omega; C_TW^{\delta})$.
\item{\textbf{For $X^{\RS{rLlr}}$}
According to Assumption \ref{ass:g}, we have 
\begin{align*}
    f(D)Y_\epsilon^{\RS{rLlr}}\to f(D)Y^{\RS{rLlr}},\qquad as \; \epsilon\to0.
\end{align*}
Then by \eqref{eq;X Y;3-order}, it suffice to show the convergence of $Z_\epsilon^{\RS{rLlr}}$. By the paraproduct estimates and the smoothing properties of the heat semigroup, $Z_\epsilon^{\RS{rLlr}}$ converges in $L^4(\Omega;C([\tau,T];W^{5/4+\delta}))$ for every $\tau\in(0,T)$. Moreover, the preceding estimates yield, for $\rho=1-\frac{3}{2\gamma}>0$ and small enough $0<\delta<\min(1/4,\frac{2\gamma-3}{4})$,
\begin{align*}
    \|Z_\epsilon^{\RS{rLlr}}(t)\|_{W^{5/4+\delta}}
    \lesssim t^\rho\|Y_\epsilon\|_{C_TW^{-\frac{1}{4}+\delta}}\|Y_\epsilon^{\RS{lr}}(0)\|_{W^{1+\delta}}.
\end{align*}
Thus $Z_\epsilon^{\RS{rLlr}}$ extends continuously to $t=0$ with value zero. Applying the same estimates to differences of the approximations yields convergence in $L^4(\Omega;C_TW^{5/4+\delta})$, and hence also in $L^2(\Omega;C_TW^{5/4+\delta})$.
}
    }
\end{itemize}

	\end{proof}

\begin{lemma}\label{le;convergence;probability}
Let  $X_\epsilon(\omega) ,Y(\omega):\Omega\to \m R$, satisfied $X_\epsilon(\omega)\to0$ in probability sense namely, for any $\delta>0$
\begin{align*}
    \lim_{\epsilon\to0}\m P(|X_\epsilon|>\delta)=0
\end{align*}
and $Y(\omega)<\infty$, $\m P$-a.s, then we have $X_\epsilon Y\to0$ in probability sense. 
	\begin{proof}
	Fix $\eta\in(0,1]$. Since $Y<\infty$ a.s., we can choose $N=N(\eta)$ sufficient large such that $\m P(Y\leq N)>1-\eta$ and $\m P(Y>N)\leq\eta$.  Then for any $\delta>0$,
		\begin{align*}
			\m P(|X_\epsilon Y|>\delta)\leq& \m P(|X_\epsilon Y|>\delta,|Y|\leq N)+\m P(|Y|>N)\\
			\leq &\m P(|X_\epsilon| >\frac{\delta}{N})+\m P(|Y|>N)\leq \m P(|X_\epsilon| >\frac{\delta}{N})+\eta.
		\end{align*}
Since $\lim_{\epsilon\to0}\m P(|X_\epsilon| >\frac{\delta}{N})=0$, therefore $\lim_{\epsilon\to0}\m P(|X_\epsilon Y|>\delta)\leq \eta$. Since $\eta>0$ is arbitrary, we conclude that $\lim_{\epsilon\to0}\m P(|X_\epsilon Y|>\delta)$.
\end{proof}
	\end{lemma}
    Finally, we finish our proof of main result. For our complete, we define $X^\tau$ $\tau\in\ml T\backslash\{\RS{n},\RS{lr},\RS{rLlr}\}$ as $X^{\tau}=B_f(X^{\tau_1},X^{\tau_2})$ where $\tau=\tau_1\cdot \tau_2$.
\begin{proof}[Proof of Theorem \ref{th;main result}]
		Let $u_0\in W^s$, $\ml T=\ml T(\RS{n})$, $b=\frac12<\gamma-1$ and $0<\delta<\min\{\frac{1}{12},\frac{\gamma}{2}-\frac{3}{4}\}$. By Proposition \ref{prop;order-1},  \ref{prop;order-2} and \ref{prop;order-3}, there exists $\frac{1}{2}$-order regular function on $\ml T$ such that 
        \begin{equation}\label{eq;X adapted}
             r(\RS{n})=\alpha>-1/4,\quad X^\tau\in C_TW^{r(\tau)},\quad for\: \tau\in \ml T(\RS{n}),\: T>0.
        \end{equation}
        Let $V_\alpha(T)=C_TW^{\alpha}$, given
        \begin{align*}
             \m X=\{Y=X^{\RS{n}},X^{\RS{lr}},X^{\RS{rLlr}}\},
        \end{align*}
    where $Y$, $X^{\RS{lr}}$ and $X^{\RS{rLlr}}$ are given by Proposition \ref{prop;order-1},  \ref{prop;order-2} and \ref{prop;order-3}. One can verify directly that $\m X$ is an enhanced data generated by $\ml T'=\{\RS{n},\RS{lr},\RS{rLlr}\}$ adapted to $(V,\ml T,r)$.  Since $2\alpha+b>0$, Theorem \ref{th;local result;gRBE;subcritical} yields that the existence of some $s>1/4$ and  ${T^*}<\bar T$  such that there exists a unique solution $(u,v)\in C_{T^*}W^\alpha\times C_{T^*}W^s$ to equation \eqref{eq;gRBE;subcritcal} driven by $(\m X,u_0)$. Moreover, $u(0)=u_0+\sum_{\tau\in\ml T'}c(\tau) X^{\tau}(0)=u_0+Y(0)$.
        \par Now let $\sigma_n$ be an arbitrary sequence converging to $0$ as $n\to\infty$.  Choose $u_0^{\sigma_n}$ and  $\m X^{\sigma_n}=\{Y_{\sigma_n},X_{\sigma_n}^{\RS{lr}},X_{\sigma_n}^{\RS{rLlr}}\}$ such that $\Vert u_0^{\sigma_n}-u_0\Vert_{W^s}\to 0$ and $\Vert \m X^{\sigma_n}-\m X\Vert_{X_T^{\alpha,b}}\to0$ as $\epsilon\to0$, then by Theorem \ref{th;local result;gRBE;subcritical}(3) and Lemma \ref{le;convergence;probability}, we obtain a sequence of existence times $\lim_{n\to\infty}T_{\sigma_n}\to T^*$ such that 
        \begin{align*}
            \lim_{n\to\infty} \Vert v^{\sigma_n}-v\Vert_{C_{T_{\sigma_n}\wedge T^*}W^s}=0
        \end{align*}
 in probability sense.   Since the sequence $\sigma_n$ is arbitrary, we conclude that
        \begin{align*}
            \lim_{\epsilon\to0}\Vert v^\epsilon-v\Vert_{C_{T_{\epsilon}\wedge T^*}W^s}=0
        \end{align*}
   in probability sense, and immediately deduce that 
\begin{align*}
            \lim_{\epsilon\to0}\Vert u^\epsilon-u\Vert_{C_{T_{\epsilon}\wedge T^*}W^{-1/4+\delta}}=0
        \end{align*}
 in probability sense, where $\alpha=-1/4+\delta$. Choosing $\ml R=\{(\RS{n},\RS{rLlr}),(\RS{rLlr},\RS{n}),(\RS{lr},\RS{lr}),(\RS{rLlr},\RS{rLlr}),(\RS{rLlr},\RS{lr}),(\RS{lr},\RS{rLlr})\}$, the analysis in Section \ref{sec;high order}shows that if $(u^\epsilon,v^\epsilon)$ is the solution of \eqref{eq;gRBE;subcritcal} driven by $(\m X,u_0)$, then \(u^\epsilon\) coincides with the solution of \eqref{eq;Burgers;app} with initial condition
\[
u^\epsilon(0)
=
u_0^\epsilon+Y^\epsilon(0)
=
\phi_\epsilon*u(0).
\]
This completes the proof of Theorem \ref{th;main result}.
      \end{proof}
    \section{Open questions}
\begin{itemize}
    \item {The first question concerns the assumptions on the Fourier multiplier $g(D)$. In the present proof, the estimate for $R^3$ relies on the evenness of $g$, which makes $f(x)=xg(x)$ odd and yields the cancellation used in the symmetrization argument. It is therefore natural to ask whether the well-posedness and convergence results remain valid for more general zero-order Fourier multipliers $g(D)$ without this symmetry assumption. Can the required estimate be obtained by a different argument, or are additional renormalization terms necessary?}
    \item {As noted in Remark \ref{rem;energy}, the present energy estimate requires $\gamma>2$ when $\alpha<0$. It remains open whether a suitable energy estimate can be established in the regime $\gamma\in(3/2,2]$ covered by the main theorem. One possible approach is to exploit the structure of the enhanced data through a controlled or paracontrolled formulation in Sobolev spaces $H^s$.}
\end{itemize}

\appendix

\section{Gaussian computation}
In this section, we give some useful computations which take crucial rules in computation of high order expansion.
\begin{proposition}\label{prop;key computation}
	Fix a countable index set $\ml Z$, and let $\{f_\kappa\}_{\kappa\in\ml Z}$ be a family of continuous stochastic processes such that $f_\kappa(t)\in\bigoplus_{k=1}^{N}\ml I_k$ for every $\kappa\in\ml Z$ and $t\in[0,T]$, where $N\in\mathbb N$ is fixed. We denote
	$$F_{\kappa\eta}(t,s)=\m Ef_{\kappa}(t)f_{\eta}(s),\quad \hat F_{\kappa\eta}(s,t)=\m E[(f_{\kappa}(t)-f_{\kappa}(s))(f_{\eta}(t)-f_{\eta}(s))].$$
	 We also assume that $\overline {f_\eta}(t)=f_{\iota \eta}(t)$ Let $\{C_\epsilon\}_{\epsilon\in(0,1]}$ be a family of functions $C_\epsilon:\ml Z\to [0,1]$ such that
	 \begin{itemize}
		\item{$C_\epsilon(\kappa)\geq C_{\bar\epsilon}(\kappa)$ for $\epsilon<\bar\epsilon$;}
		\item{Set $\{\kappa:C_\epsilon(\kappa)\neq0\}$ is finite for every $\epsilon>0$,}
		\item{$\lim_{\epsilon\to 0}C_\epsilon(\kappa)=1$ for every $\kappa\in\ml Z$. }
	\end{itemize}
	Moreover, $g_\kappa(x)$, $\kappa\in \ml Z$ satisfy
	 \begin{itemize}
		\item{$\ml F(g_\kappa(x))(k)=1$ when $k=\psi(\kappa)$ and vanishes otherwise, where $\psi:\ml Z\mapsto \mathbb{Z}$ is a surjection and $k\in \m Z$.}
		\item{there exists an involution $\iota:\ml Z\to\ml Z$ such that $g_{\iota\kappa}=\bar g_\kappa$ and $\psi(\iota\kappa)=-k$}
	\end{itemize}
	\par Let $\ml Z_*=\ml Z\backslash\{\kappa\in\ml Z:\psi(\kappa)=0\}$Define $F_\epsilon(t,x)$ defined on $[0,T]\times\m T$ be the stochastic process defined by 
	\begin{equation}\label{eq;rep;F}
		F_\epsilon(t,x)=\sum_{\kappa\in\ml Z_*}C_\epsilon(\kappa)f_\kappa(t)g_\kappa(x).
	\end{equation}
Assume that there exists $a\in\m R$ and $b\in(0,1/2]$ such that
\begin{itemize}

\item {The following estimates hold for $a,b,c\in\m R$ and $c<a$}
\end{itemize}
	\begin{align}
		S_1=&\sum_{\kappa,\eta\in\ml Z_*}\sup_{s,t\in[0,T]}\vert F_{\kappa\eta}(t,s)\vert\vert \psi(\kappa)\vert^a \vert \psi(\eta)\vert^a<\infty,\\
		S_2=&\sum_{\kappa,\eta\in\ml Z_*}|\psi(\kappa)|^c|\psi(\eta)|^c\sup_{s,t\in[0,T]}\vert \frac{\hat F_{\kappa\eta}(s,t)}{\vert t-s\vert^{2b}}<\infty.
	\end{align}
	Then there exists a stochastic process $F$ such that $F_\epsilon\to F$  in $L^2(\Omega;C_TW^\alpha)\cap L^2(\Omega;C_T^\beta W^c)$ for $\alpha<a$, $\beta< b$. 
	\begin{proof}
		
        Let $0<\bar\epsilon<\epsilon$ and $F^{\bar \epsilon}_\epsilon(x,t):=F_\epsilon(x,t)-F_{\bar\epsilon}(x,t)$, We firstly establish the convergence in Sobolev space. By the definition of $\hat F_{\kappa\eta}$ and $F_{\kappa\eta}(t,s)$, we have
               \begin{align*}
          | \hat F_{\kappa\eta}(s,t)|=|\m E[(f_{\kappa}(t)-f_{\kappa}(s))(f_{\eta}(t)-f_{\eta}(s))]|\lesssim \sup_{r,\tau\in[0,T]}|F_{\kappa\eta}(r,\tau)|=F_{\kappa\eta}.
        \end{align*}
        Denote $C_\epsilon^{\bar\epsilon}(\kappa)=C_\epsilon(\kappa)-C_{\bar \epsilon}(\kappa)$ and 
        \begin{align}
		 S_1^{\epsilon,\bar\epsilon}=&\sum_{\kappa,\eta\in\ml Z_*}C_{\epsilon}^{\bar\epsilon}(\kappa)C_{\epsilon}^{\bar\epsilon}(\iota\eta)\sup_{s,t\in[0,T]}\vert F_{\kappa\eta}(t,s)\vert\vert \psi(\kappa)\vert^a \vert \psi(\eta)\vert^a<\infty,\\
		S_2^{\epsilon,\bar\epsilon}=&\sum_{\kappa,\eta\in\ml Z_*}C_{\epsilon}^{\bar\epsilon}(\kappa)C_{\epsilon}^{\bar\epsilon}(\iota\eta)|\psi(\kappa)|^c|\psi(\eta)|^c\sup_{s,t\in[0,T]}\vert \frac{\hat F_{\kappa\eta}(s,t)}{\vert t-s\vert^{2b}}<\infty.
	\end{align}
        by discrete H\"older inequality, for any $\alpha\in[c,a)$
        \begin{equation}\label{eq;S1S2}
                  \begin{aligned}
      \sum_{\kappa,\kappa'\in\ml Z_*}&C_{\epsilon}^{\bar\epsilon}(\kappa)C_{\epsilon}^{\bar\epsilon}(\iota\kappa')|\psi(\kappa)|^{\alpha}|\psi(\kappa')|^{\alpha}|\hat F_{\kappa\kappa'}(t,s)|\\
      \lesssim& \vert t-s\vert^{2b(1-\theta)}\sum_{\kappa,\kappa'\in\ml Z_*}(C_{\epsilon}^{\bar\epsilon}(\kappa)C_{\epsilon}^{\bar\epsilon}(\iota\kappa')|\psi(\kappa)|^{a}|\psi(\kappa')|^{a}F_{\kappa\kappa'})^\theta\\
      \cdot&\vert (C_{\epsilon}^{\bar\epsilon}(\kappa)C_{\epsilon}^{\bar\epsilon}(\iota\kappa')|\psi(\kappa)|^{c}|\psi(\kappa')|^{c}|\frac{|\hat F_{\kappa\eta}(s,t)|}{\vert t-s\vert^{2b}})^{1-\theta}  \lesssim|t-s|^{2b(1-\theta)}(S_1^{\epsilon,\bar\epsilon})^\theta (S_2^{\epsilon,\bar\epsilon})^{1-\theta}
        \end{aligned}
        \end{equation}
      where $\alpha=a\theta+c(1-\theta)$ for $\theta\in[0,1)$.  Using the translation $\overline {\ml Fg_{\kappa'}}(k)=\ml Fg_{\iota \kappa'}(-k)$, therefore, 
        \begin{align*}
      \m E[|k|^{2\alpha}\vert \ml F(F_\epsilon^{\bar\epsilon}(t)-F_\epsilon^{\bar\epsilon}(s))(k)\vert^2]=&\sum_{\kappa,\kappa'\in\ml Z_*}C_{\epsilon}^{\bar\epsilon}(\kappa)C_{\epsilon}^{\bar\epsilon}(\kappa')\m E[(f_\kappa(t)-f_\kappa(s))(f_{\iota \kappa'}(t)-f_{\iota \kappa'}(s))]|k|^\alpha\ml Fg_\kappa(k)|k|^{\alpha}\ml Fg_{\iota \kappa'}(-k)\\
      =&\sum_{\kappa,\kappa'\in\ml Z_*}C_{\epsilon}^{\bar\epsilon}(\kappa)C_{\epsilon}^{\bar\epsilon}(\iota \kappa')\m E[(f_\kappa(t)-f_\kappa(s))(f_{ \kappa'}(t)-f_{\kappa'}(s))]|k|^\alpha\ml Fg_\kappa(k)|k|^{\alpha}\ml Fg_{\kappa'}(-k)\\
      =&\sum_{\kappa,\kappa'\in\ml Z_*}C_{\epsilon}^{\bar\epsilon}(\kappa)C_{\epsilon}^{\bar\epsilon}(\iota\kappa')\hat F_{\kappa\kappa'}(t,s)\textbf{1}_{\psi(\kappa)=-\psi(\kappa')=k}|\psi(\kappa)|^\alpha|\psi(\kappa')|^\alpha
        \end{align*}
        where we use the translation $\overline {\ml Fg_{\kappa'}}(k)=\ml Fg_{\iota \kappa'}(-k)$. 
Noting that the zero frequency has no contribution by the definition of $F$, Parseval's identity and \eqref{eq;S1S2} yields
		\begin{align*}
			\m E[\Vert F^{\bar{\epsilon}}_\epsilon(x,t)-F^{\bar{\epsilon}}_\epsilon(x,s)\Vert_{H^{\alpha}(\mathbb{T})}^{2p}]\sim& \m E(\sum_{k\in \m Z}|k|^{2\alpha}\vert \ml F(F_\epsilon^{\bar\epsilon}(t)-F_\epsilon^{\bar\epsilon}(s))(k)\vert^2)^p\\
			=&\sum_{k_1,k_2,\cdots,k_p\in \m Z_*}\m E\prod_{i=1}^{p}|k_i|^{2\alpha}|\ml F(F_\epsilon^{\bar\epsilon}(t)-F_\epsilon^{\bar\epsilon}(s))(k_i)\vert^2\\
			\lesssim &\sum_{k_1,k_2,\cdots,k_p\in \m Z_*}\prod_{i=1}^{p}|k_i|^{2\alpha}(\m E|\ml F(F_\epsilon^{\bar\epsilon}(t)-F_\epsilon^{\bar\epsilon}(s))(k_i)\vert^{2p})^\frac{1}{p}\\
			\lesssim&\sum_{k_1,k_2,\cdots,k_p\in \m Z_*}\prod_{i=1}^{p}|k_i|^{2\alpha}\m E|\ml F(F_\epsilon^{\bar\epsilon}(t)-F_\epsilon^{\bar\epsilon}(s))(k_i)\vert^{2}\\
            \lesssim&(\sum_{k\in \m Z_*}\sum_{\kappa,\kappa'\in\ml Z_*}C_{\epsilon}^{\bar\epsilon}(\kappa)C_{\epsilon}^{\bar\epsilon}(\iota\kappa')\hat F_{\kappa\kappa'}(t,s)\textbf{1}_{\psi(\kappa)=-\psi(\kappa')=k}|\psi(\kappa)|^\alpha|\psi(\kappa')|^\alpha)^p\\
            \lesssim&(\sum_{\kappa,\kappa'\in\ml Z_*}C_{\epsilon}^{\bar\epsilon}(\kappa)C_{\epsilon}^{\bar\epsilon}(\iota\kappa')|\hat F_{\kappa\kappa'}(t,s)||\psi(\kappa)|^\alpha|\psi(\kappa')|^\alpha)^p\\
			\lesssim&|t-s|^{2b(1-\theta)p}(S_1^{\epsilon,\bar\epsilon})^{\theta p} (S_2^{\epsilon,\bar\epsilon})^{(1-\theta)p}
		\end{align*}
        where we use the fact  $\sum_{k\in \m Z_*}\textbf{1}_{\psi(\kappa)=-\psi(\kappa')=k }=\textbf{1}_{\psi(\kappa)=-\psi(\kappa')\neq0}$. By Lebesgue’s dominated convergence theorem, we have
        \begin{equation}\label{eq;S1S2;convergence}
        \lim_{\epsilon\to0}\sup_{0<\bar\epsilon<\epsilon}S_i^{\epsilon,\bar\epsilon}=0,\qquad i=1,2.
        \end{equation}
		Choosing $p$ sufficiently large, it then follows from a straightforward modification of Kolmogorov’s continuity criterion and \eqref{eq;S1S2;convergence}, that for $\beta<b(1-\theta)-\frac{1}{2p}$
        \begin{align*}
		\lim_{\epsilon\to0}\sup_{0<\bar\epsilon<\epsilon}\m E\Vert F_\epsilon^{\bar\epsilon}(0)\Vert^2_{H^{a\theta+c(1-\theta)}}=0\qquad,\lim_{\epsilon\to0}\sup_{0<\bar\epsilon<\epsilon}\m E\Vert F_\epsilon^{\bar\epsilon}\Vert^2_{C_T^{\beta}H^{a\theta+c(1-\theta)}}=0. 
		\end{align*}
       By choosing \(\theta\) sufficiently close to \(1\) and then taking \(p\) sufficiently large, we obtain convergence in \(L^2(\Omega;C_TH^\alpha)\) for every \(\alpha<a\). Similarly, choosing \(\theta\) close to \(0\) and \(p\) sufficiently large yields convergence in \(L^2(\Omega;C_T^\beta H^c)\) for every \(0<\beta<b\), since \(a\theta+c(1-\theta)>c\). Therefore, there exists a limit process \(F\) such that
\[
F_\epsilon\to F
\quad\text{in}\quad
L^2(\Omega;C_TH^\alpha\cap C_T^\beta H^c)
\]
for every \(\alpha<a\) and \(0<\beta<b\).
	    \par To obtain the spatial H\"older regularity, we further prove the convergence in Besov spaces $B_{2p,2p}^\alpha$. By definition of $B_{2p,2p}^\alpha$ and Gaussian hypercontractivity, we obtain
        \begin{align*}
            \m E[\Vert F^{\bar{\epsilon}}_\epsilon(x,t)-F^{\bar{\epsilon}}_\epsilon(x,s)\Vert_{B_{2p,2p}^\alpha(\mathbb{T})}^{2p}]=&\sum_{l\geq -1}2^{2pl\alpha}\m E\Vert \Delta_l( F_\epsilon^{\bar\epsilon}(x,t)- F_\epsilon^{\bar\epsilon}(x,s)\Vert_{L^{2p}}^{2p}]\\
            \lesssim&\sum_{l\geq -1}2^{2pl\alpha}\Vert \m E|\Delta_l( F_\epsilon^{\bar\epsilon}(x,t)- F_\epsilon^{\bar\epsilon}(x,s)|^2\Vert_{L^p}^p.
        \end{align*}
For $l\geq0$, we have
        \begin{align*}
            \m E|\Delta_l( F_\epsilon^{\bar\epsilon}(x,t)- F_\epsilon^{\bar\epsilon}(x,s)|^2=&(2\pi)^{-2}\sum_{k,k'\in \m Z}\rho_l(k)\rho_l(-k')e^{i(k-k')x}\m E[\ml F(F_\epsilon^{\bar\epsilon}(t)-F_\epsilon^{\bar\epsilon}(s))(k)\ml F(F_\epsilon^{\bar\epsilon}(t)-F_\epsilon^{\bar\epsilon}(s))(-k')]\\
            \lesssim&\sum_{k,k'\in \m Z}\rho_l(k)\rho_{l}(-k')\sum_{\kappa,\kappa'\in\ml Z}C_\epsilon^{\bar\epsilon}(\kappa)C_\epsilon^{\bar\epsilon}(\iota\kappa')|\hat F_{\kappa\kappa'}(t,s)|\textbf{1}_{k=\psi(\kappa)}\textbf{1}_{-k'=\psi(\kappa')}\\
            \lesssim&
            \sum_{\kappa,\kappa'\in\ml Z_*}C_{\epsilon}^{\bar\epsilon}(\kappa)C_{\epsilon}^{\bar\epsilon}(\iota\kappa')|\hat F_{\kappa\kappa'}(t,s)|\textbf{1}_{|\psi(\kappa)|\sim|\psi(\kappa')|\sim2^l}.
        \end{align*}
        The case of $l=-1$ is similar. Noting that $\sum_{l\geq -1}\textbf{1}_{\psi(\kappa)\sim\psi(\kappa')\sim2^l}\lesssim\textbf{1}$,  applying the interpolation estimate \eqref{eq;S1S2}, we obtain
      \begin{align*}
             \m E[\Vert F^{\bar{\epsilon}}_\epsilon(x,t)-F^{\bar{\epsilon}}_\epsilon(x,s)\Vert_{B_{2p,2p}^\alpha(\mathbb{T})}^{2p}]
            \lesssim& (\sum_{\kappa,\kappa\in\ml Z}C_\epsilon^{\bar\epsilon}(\kappa)C_\epsilon^{\bar\epsilon}(\kappa')|\psi(\kappa)|^\alpha|\psi(\kappa')|^\alpha|\hat F_{\kappa\kappa'}(t,s)|)^p\\
            \lesssim&|t-s|^{2b(1-\theta)p}(S_1^{\epsilon,\bar\epsilon})^{\theta p} (S_2^{\epsilon,\bar\epsilon})^{(1-\theta)p}.
        \end{align*}
	  Choosing \(p\) sufficiently large, Kolmogorov’s continuity criterion and \eqref{eq;S1S2;convergence} yields, for every \(0<\eta<b(1-\theta)-\frac{1}{2p}\),
\[
\lim_{\epsilon\to0}\sup_{0<\bar\epsilon<\epsilon}\m E\Vert F_\epsilon^{\bar\epsilon}(0)\Vert^2_{B_{2p,2p}^{a\theta+c(1-\theta)}}=0,\qquad\lim_{\bar\epsilon\to0}\sup_{0<\epsilon<\bar\epsilon}
\mathbb E\left\|F_\epsilon^{\bar\epsilon}\right\|^2_
{C_T^\eta B_{2p,2p}^{a\theta+c(1-\theta)}}=0.
\]
Using the Besov embedding
\[
B_{2p,2p}^{a\theta+c(1-\theta)}
\hookrightarrow
\mathcal C^{a\theta+c(1-\theta)-\frac{1}{2p}},
\]
we choose \(\theta\) sufficiently close to \(1\) and then \(p\) sufficiently large to obtain convergence in \(L^2(\Omega;C_T\mathcal C^\alpha)\) for every \(\alpha<a\). Similarly, choosing \(\theta>0\) sufficiently close to \(0\) and then \(p\) sufficiently large yields convergence in \(L^2(\Omega;C_T^\beta\mathcal C^c)\) for every \(0<\beta<b\), since \(a\theta+c(1-\theta)>c\). Consequently,
\[
F_\epsilon\to F
\quad\text{in}\quad
L^2(\Omega;C_T\mathcal C^\alpha\cap C_T^\beta\mathcal C^c)
\]
for every \(\alpha<a\) and \(0<\beta<b\).
    \end{proof}
\end{proposition}
\section{Other useful estimate}\label{sec;useful est}
In this section, we introduce some useful estimates established in \cite{hairer2013solving} and provide several additional estimates that will be used throughout the paper.
\begin{proposition}\label{prop;useful;bound 2}[\cite{hairer2013solving}]
	 \begin{itemize}
	 	\item{(1)	The bound 
	 		\begin{align*}
	 			\int_{-\infty}^{\tau}\int_{-\infty}^{\tau'}e^{-a\vert \tau-r\vert-a\vert \tau'-r'\vert-b\vert r-r'\vert}\mathrm{d}r'\mathrm{d}r=&\frac{ae^{-b\vert \tau-\tau'\vert}-be^{-a\vert \tau-\tau'\vert}}{a(a^2-b^2)}\\
	 			\leq&\frac{1}{a(a+b)}\wedge \frac{e^{-(a\wedge b)\vert t-t'\vert}}{a\vert a-b\vert}.
	 		\end{align*}
	 		hold for $a,b>0$, $\tau,\tau'\in \mathbb{R}$.}
		\item{(2)The bound 
			\begin{align*}
				\int_{-\infty}^{t}\int_{-\infty}^{t'}&\exp(-a\vert s-s'\vert-b\vert t-s\vert-c\vert t'-s'\vert-d\vert t' -s\vert-e\vert t-s'\vert)\mathrm{d}r'\mathrm{d}r\\
				\leq&\frac{10e^{-(d\wedge e)\vert t-t'\vert}}{(b+d)(c+e)+a((b+d)\wedge(c+e))}
			\end{align*}
			hold for every $s,s'\in \mathbb{R}$ and every $a,b,c,d,e>0$.}
		\item{(3) For every $s<t$ and $u,v>0$, one has 
		\begin{align*}
			\int_s^te^{-u\vert x-s\vert-v\vert x-t\vert}\mathrm{d}x\leq \frac{4}{u+v}
			\end{align*}}
		\item{(4) For $a,b,t>0$, we bound that
		\begin{align*}
			\bigg\vert \frac{bf(at)-af(bt)}{b-a}-1\bigg\vert\leq 2abt^2 \Vert f''\Vert_{\infty}
			\end{align*}   
		Moreover, if $K$ satisfied $\sup_{y\geq0}\vert f(y)-yf'(y)\vert\leq K$, then 
		\begin{align*}
			\bigg\vert \frac{bf(at)-af(bt)}{b-a}\bigg\vert\leq K
		\end{align*}
		}
	\end{itemize}
\end{proposition}
\begin{corollary}\label{cor;ab-ba}
		For $a,b,t>0$, we bound that
	\begin{align*}
		\bigg\vert be^{-at}-ae^{-bt} -(b-a)\bigg\vert\leq (C\wedge 2abt^2 \Vert f''\Vert_{\infty})\vert b-a\vert
	\end{align*} 
	for some $C>1$
	\begin{proof}
		For $a=b$, the result obviously holds. Without loss of generality, we assume that $a>b$, using that $be^{-bt}\lesssim\frac{1}{t}$ then we have bound
		\begin{align*}
			 \vert be^{-at}-ae^{-bt}\vert\leq |b-a|e^{-at}+|a||e^{-at}-e^{-bt}|
		\end{align*}
	Then the claim is followed by Proposition \ref{prop;useful;bound 2}.
	\end{proof}
\end{corollary}
\begin{lemma}\label{le; H-hat;1}
	Let $H_{L,L'}(T,T')=\ml I_{L,L'}\ml K(T,T')$, where the operator $\ml I_{L,L'}$ is define as \eqref{eq;operator;I}. $\ml K(t,t')$ satisfies
    \begin{itemize}
        \item {$\ml K(t,t')=\ml K(t',t)$ for any $t',t\in\m R$.}
        \item { $\ml K(t-s,t'-s)=\ml K(t,t')$ for any $s\geq0$.}
    \end{itemize}
    and the following estimate 
	\begin{align*}
		\vert \ml K(t,t')\vert\leq Fe^{-a\vert t-t'\vert}
	\end{align*}
	where $a,F>0$. Let $\hat H(T,T')=H(T,T)+H(T',T')-H(T,T')-H(T',T)$, if $\vert \psi (L)\vert=\vert \psi (L')\vert\neq0$, then $\hat H$ and $H$ has estimate
	\begin{align*}
		\vert H_{L,L'}(T,T')\vert	\lesssim& \frac{1}{\vert \psi (L)\vert^{\gamma}(\vert \psi (L)\vert^{\gamma}+a)}\wedge\frac{e^{-(\vert \psi (L)\vert^\gamma\wedge a)\vert T-T'\vert}}{\vert \psi (L)\vert^{\gamma}\vert \vert \psi (L)\vert^\gamma-a\vert}\cdot F,\\
		\vert \hat H_{L,L'}(T,T')\vert\lesssim& F\frac{1\wedge \vert \psi(L)\vert^\gamma\vert T-T'\vert}{\vert \psi (L)\vert^{\gamma}(\vert \psi (L)\vert^\gamma+a)}.
	\end{align*}
    When $a=0$, the bounds for $H$ and $\hat H$ remain valid under the sole assumption $|\mathcal K(t,t')|\leq F$, without requiring conditions
	\begin{proof}
    Recall that 
    \begin{align*}
        \ml I_{L,L'}\ml K(t,t')=\int_{-\infty}^{t}\int_{-\infty}^{t'} e^{-\vert \psi(L)\vert^\gamma(t-s)}e^{-\vert \psi(L')\vert^\gamma(t'-s')} \ml K(s,s')\mathrm{d}s'\mathrm{d}s,
    \end{align*}
		Without loss of generality, we assume that $T\geq T'$. By the assumption of $\ml K(t,t')$  we can write $H_{L,L'}(T,T')$ as follows
		\begin{align*}
			H_{L,L'}(T,T')=&\int_{-\infty}^{T}\int_{-\infty}^{T'}e^{(-\vert \psi (L)\vert^\gamma(T+T'-t-t'))}\ml K(t,t')\mathrm{d}t'\mathrm{d}t\\
			=&\int_{-\infty}^{T-T'}\int_{-\infty}^{0}e^{-\vert \psi (L)\vert^\gamma(T-T'-t-t')}\ml K(t+T',t'+T')\mathrm{d}t'\mathrm{d}t\\
			=&\int_{-\infty}^{\delta}\int_{-\infty}^{0}e^{-\vert \psi (L)\vert^\gamma(\delta-t-t')}\ml K(t,t')\mathrm{d}t'\mathrm{d}t\\
			=:&H(\delta),\quad \delta=T-T'.
		\end{align*}
		Applying Proposition \ref{prop;useful;bound 2}, we obtain
		\begin{align*}
		\vert H(\delta)\vert	\leq \frac{1}{\vert \psi (L)\vert^\gamma(\vert \psi (L)\vert^\gamma+a)}\wedge\frac{e^{-(\vert \psi (L)\vert^\gamma\wedge a)\delta}}{\vert \psi (L)\vert^\gamma\vert  \vert \psi (L)\vert^\gamma-a\vert}\cdot F.
		\end{align*}
		Since the expression is symmetric with respect to $T$ and $T'$, the estimate for $|H_{L,L'}(T,T')|$ holds.
		\par For $\hat H_{L,L'}(T,T')$, by a direct calculation, we have
				\begin{align*}
			\vert \hat H_{L,L'}(T,T')\vert=&\vert 2 H_{L,L'}(T,T')-H_{L,L'}(T,T)-H_{L,L'}(T',T')\vert=2\vert H_{L,L'}(\delta)-H_{L,L'}(0)\vert\\
			\lesssim&2\vert  \int_{-\infty}^{\delta}\int_{-\infty}^{0} e^{-\vert \psi (L)\vert^\gamma(\delta-t-t')}\ml K(t,t')\mathrm{d}t'\mathrm{d}t-\int_{-\infty}^{0}\int_{-\infty}^{0} e^{-\vert \psi (L)\vert^\gamma(-t-t')}\ml K(t,t')\mathrm{d}t'\mathrm{d}t\vert\\
			\lesssim&2\vert \int_{-\infty}^{0}\int_{-\infty}^{0} e^{-\vert \psi (L)\vert^\gamma(\delta-t-t')}\ml K(t,t')-e^{-\vert \psi (L)\vert^\gamma(-t-t')}\ml K(t,t')\mathrm{d}t'\mathrm{d}t\vert\\
			&+2\vert \int_{0}^{\delta}\int_{-\infty}^{0} e^{-\vert \psi (L)\vert^\gamma(\delta-t-t')}\ml K(t,t')\mathrm{d}t'\mathrm{d}t\vert
			=\hat H_1+\hat H_2.
		\end{align*}
		For the first part $\hat H_1$, using $|1-e^{-x}|\lesssim 1\wedge x$, by Proposition \ref{prop;useful;bound 2}, we obtain
		\begin{align*}
			\hat H_1\leq 2 |e^{-\vert \psi(L)\vert^\gamma\vert \delta}-1|H_{L,L'}(0)\leq F\frac{1\wedge \vert \psi(L)\vert^\gamma\vert \delta\vert}{\vert \psi (L)\vert^\gamma(\vert \psi (L)\vert^\gamma+a)}.
		\end{align*}
		For the second part, since $t'\leq 0$, we have
		\begin{align*}
			\int_{-\infty}^{0}e^{-\vert \psi (L)\vert^\gamma(\delta-t-t')}\ml K(t,t')\mathrm{d}t'\leq F\cdot\frac{e^{-\vert \psi(L)\vert^\gamma (\delta-t)-at}}{\vert \psi(L)\vert^\gamma+a}.
		\end{align*}
Hence, using $|1-e^{-x}|\lesssim 1\wedge x$, we have 
   \begin{align*}
			\hat H_2\lesssim F\frac{1\wedge|\psi(L)|^\gamma|\delta}{|\psi(L)|^\gamma(|\psi(L)|^\gamma+a)}.
		\end{align*}
Combining the above estimates yields
\begin{align*}
    \vert \hat H_{L,L'}(T,T')\vert\lesssim& F\frac{1\wedge \vert \psi(L)\vert^\gamma\vert \delta\vert}{\vert \psi (L)\vert^\gamma(\vert \psi (L)\vert^\gamma+a)}.
\end{align*}
\par For $a=0$, the same bounds follow directly from $|\mathcal K(t,t')|\leq F$. Indeed, writing $q=|\psi(L)|^\gamma$ and $K_t(s)=\mathbf1_{\{s\leq t\}}e^{-q(t-s)}$, we have
\[
\|K_t\|_{L^1}=q^{-1},\qquad \|K_t-K_{t'}\|_{L^1}=2q^{-1}(1-e^{-q|t-t'|}).
\]
Consequently,
\[
|H(t,t')|\leq Fq^{-2},\qquad |\hat H(t,t')|\leq F\|K_t-K_{t'}\|_{L^1}^2\lesssim Fq^{-2}(1\wedge q|t-t'|).
\]
Neither symmetry nor stationarity of $\mathcal K$ is needed in this case.
	\end{proof}
\end{lemma}
\begin{lemma}\label{le;f}
Let $g$ be even and satisfy Assumption \ref{ass:g},
and let $f(x)=xg(x)$. Then
\[
|f(x)+f(y)|\lesssim |x+y|,
\qquad x,y\in\mathbb R.
\]

\begin{proof}
Since $g$ is bounded, $f(0)=0$ and
$|f(x)|\lesssim |x|$. Thus, if $xy\geq0$,
\[
|f(x)+f(y)|
\lesssim |x|+|y|=|x+y|.
\]

If $xy<0$, by symmetry we may assume that $x>0>y$.
Since $g$ is even, we have
\[
\begin{aligned}
|f(x)+f(y)|
&=|xg(x)-(-y)g(-y)|\\
&\lesssim |x+y|+\min(x,-y)|g(x)-g(-y)|.
\end{aligned}
\]
By the mean value theorem, for some $\tau$ between $x$ and $-y$,
\[
\begin{aligned}
\min(x,-y)|g(x)-g(-y)|
&=\min(x,-y)|g'(\tau)||x+y|\\
&\leq |\tau g'(\tau)||x+y|
\lesssim |x+y|.
\end{aligned}
\]
Combining these estimates proves the claim.
\end{proof}
\end{lemma}
We now use Lemma \ref{le;f} to establish the cancellation estimate needed for the contribution associated with $P_3$.
\begin{lemma}\label{le;package}
Let $\gamma\in[1,2]$ and $t\geq s$. Let $f(x)=xg(x)$ and $g(x)$ satisfy Assumption \ref{ass:g} 
Define
\[
\ml H_{k,m}(t,s)
=f(m+k)e^{-(|k-m|^\gamma+|k|^\gamma)(t-s)}
+f(m-k)e^{-(|k+m|^\gamma+|k|^\gamma)(t-s)}.
\]
Then there exists a constant $c>0$ such that, for any
$0\leq\theta\leq\gamma-1$,
\[
|\ml H_{k,m}(t,s)|
\lesssim
e^{-c(|k+m|^\gamma+|m|^\gamma)|t-s|}
|m|^\theta (|k+m|+|k|)^{1-\theta},
\]
for any $k,m\in\mathbb Z_*$ with $k\neq m$.
\begin{proof}[Proof of Lemma \ref{le;package}]
The cases $k=0$ or $m=0$ follow directly, so we assume
that $km\neq0$.
By a direct computation, we have
\[
\begin{aligned}
\ml H_{k,m}(t,s)
={}&m\Big(e^{-(|k-m|^\gamma+|k|^\gamma)(t-s)}
+e^{-(|k+m|^\gamma+|k|^\gamma)(t-s)}\Big)\\
&+k\Big(e^{-(|k-m|^\gamma+|k|^\gamma)(t-s)}
-e^{-(|k+m|^\gamma+|k|^\gamma)(t-s)}\Big),
\end{aligned}
\]
where we have used the identity
\[
ac+bd=\frac12((a+b)(c+d)+(a-b)(c-d)).
\]
Noting that
$|e^{-x}-e^{-y}|\lesssim(1\wedge|x-y|)e^{-(x\wedge y)}$
for $x,y\geq0$, we obtain
\begin{equation}\label{eq;est;H}
    |\ml H_{k,m}(t,s)|
\lesssim
e^{-(||k|-|m||^\gamma+|k|^\gamma)(t-s)}
\Big(
|m|+|k|\big(1\wedge
((|k|^{\gamma-1}|m|+|k||m|^{\gamma-1})(t-s))\big)
\Big),
\end{equation}
where we have used
\[
\big||k-m|^\gamma-|k+m|^\gamma\big|
\lesssim |k|^{\gamma-1}|m|+|k||m|^{\gamma-1}.
\]
Moreover, for some $c>0$,
\[
||k|-|m||^\gamma+|k|^\gamma
\geq c(|k+m|^\gamma+|k|^\gamma),
\]
and
\[
|k+m|^\gamma+|k|^\gamma
\asymp |m|^\gamma+|k|^\gamma.
\]
Hence, using $xe^{-x}\lesssim1$ and decreasing $c$
if necessary, we further obtain
\[
|\ml H_{k,m}(t,s)|
\lesssim
e^{-c(|k+m|^\gamma+|k|^\gamma)(t-s)}
\left(
|m|+
\frac{|k|^\gamma|m|+|k|^2|m|^{\gamma-1}}
{|m|^\gamma+|k|^\gamma}
\right).
\]
For $\gamma\in[1,2]$, we have
\begin{align*}
    |k|^\gamma|m|+|k|^2|m|^{\gamma-1}\lesssim |k||m|^\gamma+|k|^{2\gamma-1}|m|^{2-\gamma}+|k|^{\gamma-1}|m|^2+|k|^{\gamma+1}.
\end{align*}
and similarly
\begin{align*}
    |k|^\gamma|m|+|k|^2|m|^{\gamma-1}\lesssim |m||k|^\gamma+|m|^{2\gamma-1}|k|^{2-\gamma}+|m|^{\gamma-1}|k|^2+|m|^{\gamma+1},
\end{align*}
where we use the fact that all monomials have the same total
degree and that the powers of $|k|$(or $|m|$) on the left-hand side lie
between the smallest and largest powers of $|k|$(or $|m|$) on the
right-hand side.
Consequently,
\[
\frac{|k|^\gamma|m|+|k|^2|m|^{\gamma-1}}
{|m|^\gamma+|k|^\gamma}
\lesssim
(|k|^{2-\gamma}+|m|^{2-\gamma})
(|k|\wedge|m|)^{\gamma-1}.
\]
Using
\[
|k+m|^\gamma+|k|^\gamma
\asymp |k+m|^\gamma+|m|^\gamma,
\]
it follows that
\[
|\ml H_{k,m}(t,s)|
\lesssim
e^{-c(|k+m|^\gamma+|m|^\gamma)(t-s)}
\Big(
|m|+
(|k|^{2-\gamma}+|m|^{2-\gamma})
(|k|\wedge|m|)^{\gamma-1}
\Big),
\]
after decreasing $c$ if necessary.

Finally, for every $0\leq\theta\leq\gamma-1$, the
expression in parentheses is bounded by a constant multiple of
\[
|m|^\theta(|k+m|+|k|)^{1-\theta}.
\]
Indeed, when $|k|\leq|m|$, it is bounded by $C|m|$,
and $|k+m|+|k|\geq|m|$. When $|k|\geq|m|$, we use
\[
|k|^{2-\gamma}|m|^{\gamma-1}
\leq |m|^\theta|k|^{1-\theta}.
\]
Therefore,
\[
|\ml H_{k,m}(t,s)|
\lesssim
e^{-c(|k+m|^\gamma+|m|^\gamma)(t-s)}
|m|^\theta(|k+m|+|k|)^{1-\theta},
\]
which proves the lemma for $f(k)=k$. Now for a general case $f(k)$, noting that $\ml H_{k,m}(t,s)$ can be written as 
\[
\begin{aligned}
\ml H_{k,m}(t,s)
={}&(\frac{f(m+k)+f(m-k)}{2})\Big(e^{-(|k-m|^\gamma+|k|^\gamma)(t-s)}
+e^{-(|k+m|^\gamma+|k|^\gamma)(t-s)}\Big)\\
&+(\frac{f(m+k)-f(m-k)}{2})\Big(e^{-(|k-m|^\gamma+|k|^\gamma)(t-s)}
-e^{-(|k+m|^\gamma+|k|^\gamma)(t-s)}\Big).
\end{aligned}
\]
Since $f$ is odd, then by Lemma \ref{le;f},  we can also obtain \eqref{eq;est;H}, then the general case is similar as the $f(x)=x$, which completes the proof.
\end{proof}
\end{lemma}

We now use the techniques introduced above to derive key estimates for the Gaussian trees $X^{\RS{rLlr}}$
\begin{proposition}\label{prop;resonant convergence}
	Let $L=(k,l,m),L'=(k',l',m')\in\m Z_*^3$, $T,T'\geq0$, and $X\otimes Y(t,t')=X(t)Y(t')$. Let $f$ satisfy Assumption \ref{ass:g} Denote 
    \begin{align*}
    R_{L,L'}(T,T'):=\m E\bigg(Y_k\otimes Y_k'\cdot f(\psi(L_\downarrow))f(\psi(L'_\downarrow))\ml I_{L_{\downarrow},L'_{\downarrow}} ((Y_lY_m)\otimes (Y_{l'}Y_{m'}))(t,t')\bigg)  
    \end{align*}
    where $L_\downarrow=(l,m)$, $L'_\downarrow=(l',m')$,$\psi((a,b)))=a+b$. Define an involution $S$ on $\m Z_*^3$ by 
    \begin{equation}\label{eq;def;S}
        S(k,l,m)=(l,k,m).
    \end{equation}
 Then it can be written in three class
    \begin{equation}\label{eq;R;decomposition}
        \sum_{\substack{N\in\mk S L\\N'\in\mk S L'}}R_{N,N'}(t,t')=\sum_{i=1}^3R^i_{L,L'}(t,t')
    \end{equation}
where 
\begin{align*}
   R^1_{L,L'}(t,t')=&2\sum_{\substack{N\in\mk S L\\N'\in\mk S L'}}R_{N,N'}(P_1;t,t'),\\
 R^2_{L,L'}(t,t')=&4\sum_{\substack{N\in\mk S L\\N'\in\mk S L'}}R_{N,N'}(P_2;t,t'),\\
  R^3_{L,L'}(t,t')=&\sum_{\substack{N\in\mk S L\\N'\in\mk S L'}}\Bigl(R_{N,N'}(P_3;t,t')+R_{SN,N'}(P_3;t,t')+R_{N,SN'}(P_3;t,t')+R_{SN,SN'}(P_3;t,t')\Bigr).
\end{align*}
	and $R_{L,L'}(P_i;T,T')$, $i=1,2,3$ satisfy the following estimates
	\begin{align*}
		R_{ L, L'}(P_1;t,t')\lesssim\frac{\vert (l+m)\vert^2\vert k\vert^{2\sigma-\gamma}\vert l\vert^{2\sigma-\gamma}\vert m\vert^{2\sigma-\gamma}e^{-\vert k\vert^\gamma\vert t-t'\vert}}{\vert l+m\vert^\gamma(\vert l+m\vert^\gamma+(\vert l\vert^\gamma+\vert m\vert^\gamma ))}
	\end{align*}
	and 
	\begin{align*}
		R_{ L, L'}(P_2;t,t')\lesssim  \frac{\vert (l+m)\vert\vert k+m\vert\vert k\vert^{2\sigma-\gamma}\vert l\vert^{2\sigma-\gamma}\vert m\vert^{2\sigma-\gamma}e^{-(\vert k\vert^\gamma\wedge\vert l\vert^\gamma)\vert t-t'\vert}}{(\vert l+m\vert^\gamma+\vert l\vert^\gamma+|m|^\gamma)(\vert k+m\vert^\gamma+\vert k\vert^\gamma)\wedge(\vert l+m\vert^\gamma+\vert l\vert^\gamma)(\vert k+m\vert^\gamma+\vert k\vert^\gamma+|m|^\gamma)}.
	\end{align*}

For $R_{L,L'}(P_3;t,t')$, we have
\begin{equation}\label{eq;R;3}
\begin{aligned}
   \Bigl(R_{L,L'}(P_3;t,t')+&R_{SL,L'}(P_3;t,t')+R_{L,SL'}(P_3;t,t')+R_{SL,SL'}(P_3;t,t')\Bigr)\\
    &=|kk'm|^{2\sigma-\gamma}\int_{-\infty}^t\int_{-\infty}^{t'}\ml H_{k,m}(t,s)\ml H_{k',-m}(t',s')e^{-|m|^\gamma|s-s'|}\mathrm ds'\mathrm ds.
\end{aligned}
\end{equation}
where 
\begin{align*}
     \ml H_{k,m}(t,s)=f(m+k)e^{-(\vert k+m\vert^\gamma+\vert k\vert^\gamma)(t-s)}+f(m-k)e^{-(\vert k-m\vert^\gamma+\vert k\vert^\gamma)(t-s)}.
\end{align*}
	\begin{proof}
		\par Rewrite $R_{L,L'}(T,T')$ as 
		\begin{equation}\label{eq;R}
		\int_{-\infty}^{t}\int_{-\infty}^{t'}e^{-\vert l+m\vert^\gamma(t-s)}e^{-\vert l'+m'\vert^\gamma(t'-s')}f(l+m)f(l'+m')\m E[Y_k(t)Y_l(s)Y_m(s)Y_{k'}(t')Y_{l'}(s')Y_{m'}(s')],
		\end{equation}
We first expand the covariance by Wick's formula. Since $f(0)=0$, By symmetry under the exchanges $l\leftrightarrow m$ and $l'\leftrightarrow m'$, the nonzero Wick pairings can be
grouped into three classes.
\[
\begin{aligned}
P_1:&\ (k,k'),(l,l'),(m,m'),\\
P_2:&\ (k,l'),(l,k'),(m,m'),\\
P_3:&\ (k,l),(l',k'),(m,m').
\end{aligned}
\]
then we have
\begin{align*}
    \sum_{\substack{N\in\mk S L\\N'\in\mk S L'}}R_{N,N'}(t,t')
    =&\sum_{\substack{N\in\mk S L\\N'\in\mk S L'}}\Bigl(2R_{N,N'}(P_1;t,t')+4R_{N,N'}(P_2;t,t')+4R_{N,N'}(P_3;t,t')\Bigr)\\
    =&\sum_{\substack{N\in\mk S L\\N'\in\mk S L'}}\Bigl(2R_{N,N'}(P_1;t,t')+4R_{N,N'}(P_2;t,t')\\
    &\qquad+R_{N,N'}(P_3;t,t')+R_{SN,N'}(P_3;t,t')+R_{N,SN'}(P_3;t,t')+R_{SN,SN'}(P_3;t,t')\Bigr).
\end{align*}
Here $(a,b)$ denotes the contraction between the corresponding Gaussian variables $Y_a$ and $Y_b$.	By the covariance for the stationary Gaussian field $Y$, the contributions associated with the above pairings are given by
		\begin{align*}
			\vert k\vert^{\gamma-2\sigma}\vert l\vert^{\gamma-2\sigma}\vert m\vert^{\gamma-2\sigma}\ml I_1=&e^{-[\vert k\vert^\gamma\vert t-t'\vert+(\vert l\vert^\gamma+\vert m\vert^\gamma )\vert s-s'\vert]}\textbf{1}_{k+k'=0}\textbf{1}_{l+l'=0}\textbf{1}_{m+m'=0}\\
			\vert k\vert^{\gamma-2\sigma}\vert l\vert^{\gamma-2\sigma}\vert m\vert^{\gamma-2\sigma}\ml I_2=&e^{-[\vert k\vert^\gamma|t-s'|+\vert l\vert^\gamma|t'-s|+\vert m\vert^\gamma\vert s-s'\vert]}\textbf{1}_{k+l'=0}\textbf{1}_{k'+l=0}\textbf{1}_{m+m'=0}\\
			\vert k\vert^{\gamma-2\sigma}\vert k'\vert^{\gamma-2\sigma}\vert m\vert^{\gamma-2\sigma}\ml I_3=&e^{-[\vert k\vert^{\gamma}|t-s|+\vert k'\vert^\gamma|t'-s'|+\vert m\vert^\gamma\vert s-s'\vert]}\textbf{1}_{k+l=0}\textbf{1}_{k'+l'=0}\textbf{1}_{m+m'=0}
		\end{align*}
     For $i=1,2,3$, we denote by $R_{L,L'}(P_i;t,t')$ the contribution to $R_{L,L'}(t,t')$ corresponding to the pairing $P_i$. Here $L\sim\bar L$ means that $\bar L=(k,l,m)$ or $\bar L=(k,m,l)$, and the following sum counts both choices, even when they coincide. Grouping all nonzero Wick pairings under these exchanges.
	Now we remain to show the estimate for $R_{L,L'}(P_i;t,t')$, $i=1,2,3$. 
    \begin{itemize}
        \item { \textbf{For the contribution corresponding to pairing $P_1$}
        \par Since $|l+m|=|l'+m'|$, by the assumption $|f(l+m)|\lesssim |l+m|$ and Lemma \ref{le; H-hat;1}, we obtain
		\begin{align*}
			R_{L,L'}(P_1;t,t')\lesssim  & \frac{e^{-\vert k\vert^\gamma\vert t-t'\vert}}{\vert l+m\vert^\gamma(\vert l+m\vert^\gamma+(\vert l\vert^\gamma+\vert m\vert^\gamma ))}\vert (l+m)\vert^2\vert k\vert^{2\sigma-\gamma}\vert l\vert^{2\sigma-\gamma}\vert m\vert^{2\sigma-\gamma}
		\end{align*}
     
        }
        \item {	\textbf{For the contribution corresponding to pairing $P_2$}
        \par Proposition \ref{prop;useful;bound 2} yields
		\begin{align*}
			 R_{L,L'}(P_2;t,t')\lesssim&\vert f(l+m)\vert\vert f(l'+m')\vert\int_{-\infty}^{t}\int_{-\infty}^{t'}e^{-\vert l+m\vert^\gamma|t-s|}e^{-\vert k+m\vert^\gamma|t'-s'|}\ml I_2\mathrm{d}s\mathrm{d}s'\\
			=&\vert f(l+m)\vert\vert f(l'+m')\vert |k|^{2\sigma-\gamma}|l|^{2\sigma-\gamma}|m|^{2\sigma-\gamma}\\
            &\cdot\int_{-\infty}^{t}\int_{-\infty}^{t'}e^{-\vert l+m\vert^\gamma|t-s|}e^{-\vert k+m\vert^\gamma(t'-s')-\vert k\vert^\gamma|t-s'|-\vert l\vert^\gamma|t'-s|-\vert m\vert^\gamma\vert s-s'\vert}\mathrm{d}s\mathrm{d}s'\\
			\lesssim&
            \frac{\vert (l+m)\vert\vert l'+m'\vert\vert k\vert^{2\sigma-\gamma}\vert l\vert^{2\sigma-\gamma}\vert m\vert^{2\sigma-\gamma}e^{-(\vert k\vert^\gamma\wedge\vert l\vert^\gamma)\vert t-t'\vert}}{(\vert l+m\vert^\gamma+\vert l\vert^\gamma)(\vert k+m\vert^\gamma+\vert k\vert^\gamma)+\vert m\vert^\gamma(\vert l+m\vert^\gamma+\vert l\vert^\gamma)\wedge(\vert k+m\vert^\gamma+\vert k\vert^\gamma)}\\
			\lesssim&
            \frac{\vert (l+m)\vert\vert l'+m'\vert\vert k\vert^{2\sigma-\gamma}\vert l\vert^{2\sigma-\gamma}\vert m\vert^{2\sigma-\gamma}e^{-(\vert k\vert^\gamma\wedge\vert l\vert^\gamma)\vert t-t'\vert}}{(\vert l+m\vert^\gamma+\vert l\vert^\gamma+|m|^\gamma)(\vert k+m\vert^\gamma+\vert k\vert^\gamma)\wedge(\vert l+m\vert^\gamma+\vert l\vert^\gamma)(\vert k+m\vert^\gamma+\vert k\vert^\gamma+|m|^\gamma)}.
		\end{align*}
      }
      \item {\textbf{For the contribution corresponding to the pairing $P_3$}
      \par A direct estimate of the integral does not yield a suitable bound. We therefore exploit the involution $S(k,l,m)=(l,k,m)$. Define $L\overset{P_3}{\sim}L'$ if
\begin{align*}
    k+l=0,\qquad k'+l'=0,\qquad m+m'=0.
\end{align*}
The pairing condition is preserved by applying $S$ to either index. Consequently, reindexing yields
\begin{equation}\label{eq;P_3;sum;decom}
    \sum_{\substack{L,L'\in\m Z_*^3\\L\overset{P_3}{\sim}L'}}f_{L,L'}=\frac14\sum_{\substack{L,L'\in\m Z_*^3\\L\overset{P_3}{\sim}L'}}\bigl(f_{L,L'}+f_{SL,L'}+f_{L,SL'}+f_{SL,SL'}\bigr).
\end{equation}
This identity holds for absolutely convergent sums and for finite truncations invariant under these exchanges. Define
\begin{equation}\label{eq;bar R}
    \bar R^3_{L,L'}(t,t')=R^3_{L,L'}(t,t')+R^3_{SL,L'}(t,t')+R^3_{L,SL'}(t,t')+R^3_{SL,SL'}(t,t').
\end{equation}
Let
\begin{equation}\label{eq;def;H}
    \ml H_{k,m}(t,s)=f(m+k)e^{-(|k+m|^\gamma+|k|^\gamma)(t-s)}+f(m-k)e^{-(|k-m|^\gamma+|k|^\gamma)(t-s)},\qquad t\geq s.
\end{equation}
For $L=(k,-k,m)$ and $L'=(k',-k',-m)$, we have $SL=(-k,k,m)$ and $SL'=(-k',k',-m)$. Substituting these indices into \eqref{eq;bar R}, we obtain
\begin{align*}
    \bar R^3_{L,L'}(t,t')=|kk'm|^{2\sigma-\gamma}\int_{-\infty}^t\int_{-\infty}^{t'}\ml H_{k,m}(t,s)\ml H_{k',-m}(t',s')e^{-|m|^\gamma|s-s'|}\mathrm ds'\mathrm ds.
\end{align*}}
    \end{itemize}

	\end{proof}
\end{proposition}

\bibliographystyle{alpha} 
\bibliography{ref}
\end{document}